\documentclass[a4paper,12pt,leqno]{amsart}
\usepackage{latexsym}
\usepackage[all]{xy}

\usepackage{amsmath, amsfonts, amssymb,amscd,graphics,graphicx,color,eucal,setspace,xypic} 
\usepackage{comment}
\usepackage{mathtools}

\usepackage{graphicx}
\usepackage{stmaryrd}
\usepackage{bigdelim, multirow}
\usepackage[all]{xy}
\usepackage{tikz-cd}

\definecolor{gray}{gray}{0.7}
\definecolor{Gray}{gray}{0.3}

\usepackage{amscd}

\numberwithin{equation}{section}

\theoremstyle{break}
 \newtheorem{theorem}{Theorem}[section]
 \newtheorem{proposition}[theorem]{Proposition}
 \newtheorem{corollary}[theorem]{Corollary}
 \newtheorem{lemma}[theorem]{Lemma}

 \theoremstyle{definition}
 \newtheorem{definition}[theorem]{Definition}
 \newtheorem{remark}[theorem]{Remark}
 \newtheorem{example}[theorem]{Example}

\allowdisplaybreaks[4]

\def\C{\mathbb C}
\def\R{\mathbb R}
\def\Q{\mathbb Q}
\def\Z{\mathbb Z}

\def\NN{\mathcal{N}}
\def\g{\mathfrak{g}}
\def\b{\mathfrak{b}}
\def\t{\mathfrak{t}}
\def\uu{\mathfrak{u}}
\def\p{\mathfrak{p}}
\def\a{\mathfrak{a}}
\def\Levi{L}
\def\q{q}
\def\rr{r}
\def\RR{\mathcal{R}}
\def\mm{m}
\def\KK{\mathcal{K}}
\def\kk{k}
\def\AA{\mathcal{A}}
\def\der{\psi}

\DeclareMathOperator{\height}{ht}
\DeclareMathOperator{\Der}{Der}
\DeclareMathOperator{\Sym}{Sym}
\DeclareMathOperator{\Hom}{Hom}
\DeclareMathOperator{\id}{id}
\DeclareMathOperator{\Poin}{Poin}
\DeclareMathOperator{\Ad}{Ad}
\DeclareMathOperator{\ad}{ad}

\DeclareMathOperator{\SL}{SL}
\DeclareMathOperator{\Aut}{Aut}
\DeclareMathOperator{\Hess}{Hess}
\newcommand{\Flag}{Fl}

\begin{document}
  
\title[Regular nilpotent partial Hessenberg varieties]{Regular nilpotent partial Hessenberg varieties}
\author[T. Horiguchi]{Tatsuya Horiguchi}
\address{National Institute of Technology, Akashi College, 679-3, Nishioka, Uozumi-cho, Akashi, Hyogo 674-8501, Japan}
\email{tatsuya.horiguchi0103@gmail.com}



\begin{abstract}
Let $G$ be a complex semisimple linear algebraic group.
Fix a subset $\Theta$ of simple roots. 
Given a lower ideal $I$ in positive roots, one can define the regular nilpotent Hessenberg variety $\Hess(N,I)$ in the full flag variety $G/B$. 
For a $\Theta$-ideal $I$ (which is a special lower ideal), we can define the regular nilpotent partial Hessenberg variety $\Hess_\Theta(N,I)$ in the partial flag variety $G/P$. 
In this manuscript we first provide a summand formula and a product formula for the Poincar\'e polynomial of regular nilpotent partial Hessenberg varieties. 
It is a well-known result from Bernstein--Gelfand--Gelfand that the cohomology ring of the partial flag variety $G/P$ is isomorphic to the invariants in the cohomology ring of the full flag variety $G/B$ under an action of the parabolic Weyl group $W_\Theta$ generated by $\Theta$. 
We generalize this result to regular nilpotent partial Hessenberg varieties. 
More concretely, we give an isomorphism between the cohomology ring of a regular nilpotent partial Hessenberg variety $\Hess_\Theta(N,I)$ and the $W_\Theta$-invariant subring of the cohomology ring of the regular nilpotent Hessenberg variety $\Hess(N,I)$. 
Furthermore, we provide a description of the cohomology ring for a regular nilpotent partial Hessenberg variety $\Hess_\Theta(N,I)$ in terms of the $W_\Theta$-invariants in the logarithmic derivation module of the ideal arrangement $\AA_I$, which is a generalization of the result by Abe--Masuda--Murai--Sato with the author.
\end{abstract}

\maketitle

\setcounter{tocdepth}{1}

\tableofcontents

\section{Introduction}
\label{sect:Intro}

Hessenberg varieties are subvarieties of (full) flag varieties introduced by F. De Mari, C. Procesi, and M. A. Shayman in \cite{dMPS, dMaSh}. 
This subject lies in a fruitful intersection of algebraic geometry, combinatorics, topology, and
representation theory.
Particular examples are Springer fibers, Peterson varieties, and toric varieties associated with weight polytopes. 
We refer the reader to \cite{AbHo} for a survey of Hessenberg varieties. 
Two specific classes of Hessenberg varieties, called regular nilpotent and regular semisimple Hessenberg varieties, are related to other research areas such as hyperplane arrangements and graph theory. 
In fact, Tymoczko constructed an action of symmetric groups on the cohomology of regular semisimple Hessenberg varieties in type $A$, which is called the dot action \cite{Tym08}.
Shareshian and Wachs conjectured in \cite{ShWa16} a relationship between the dot action and a graded version of the Stanley's chromatic symmetric functions of incomparability graphs of natural unit interval orders.
This conjecture was proved by Brosnan and Chow \cite{BrCh}, and soon after, Guay--Paquet gave an alternative proof \cite{Gua}.
On the other hand, the cohomology rings of regular nilpotent Hessenberg varieties can be described in terms of the logarithmic derivation modules of ideal arrnagements \cite{AHMMS}. 
By using this relationship, we obtain an explicit presentation for the cohomology rings of regular nilpotent Hessenberg varieties \cite{AHMMS, EHNT19}. 
We remark that the cohomology rings of regular nilpotent Hessenberg varieties in type $A$ was explicitly described in \cite{AHHM} by a localization technique.
Furthermore, it is known that the cohomology ring of regular nilpotent Hessenberg varieties is the trivial part of the dot action on the cohomology ring of regular semisimple Hessenberg varieties \cite{AHHM, AHMMS, BaCr, BrCh}.
We remark that most generalization of this result is given by \cite{BaCr}. 

Kiem and Lee generalized the now-settled Shareshian--Wachs conjecture to regular semisimple partial Hessenberg varieties, which are subvarieties of partial flag varieties \cite{KiLe}.
It is natural to ask how we can generalize the relation with hyperplane arranegements to regular nilpotent partial Hessenberg varieties. 
In this paper we describe the cohomology rings of regular nilpotent partial Hessenberg varieties in terms of the invariants of the parabolic Weyl group action on the logarithmic derivation modules of ideal arrnagements.
For this purpose, we connect the cohomology ring of regular nilpotent partial Hessenberg varieties with that of regular nilpotent Hessenberg varieties.
We also give a summation formula and a product formula for the Poincar\'e polynomial of regular nilpotent partial Hessenberg varieties. 
Throughout this manuscript, all cohomology will be taken with rational coefficients unless
otherwise specified.

Let $G$ be a semisimple linear algebraic group over $\C$. 
Fix a Borel subgroup $B$ of $G$ and a maximal torus $T$ of $G$ in $B$.
The Lie algebras of $T \subset B \subset G$ are denoted by $\t \subset \b \subset \g$.
We write $\Phi$ for the root system of $\t$ in $\g$, and $\Phi^+$ and $\Delta$ stand for the set of positive roots and the set of simple roots in $\Phi$, respectively. 
If $I$ is a lower ideal in $\Phi^+$, i.e. a lower closed subset of $\Phi^+$ with respect to the usual partial order $\preceq$ on $\Phi^+$, then $H_I = \b \oplus (\bigoplus_{\alpha \in I} \g_{-\alpha})$ is a $\b$-submodule of $\g$ containing $\b$. 
Here, $\g_\alpha$ denotes the root space associated to a root $\alpha \in \Phi$.
Conversely, any $\b$-submodule of $\g$ containing $\b$ is of the form $H_I$ for some lower ideal $I \subset \Phi^+$.
For $x \in \g$ and a lower ideal $I \subset \Phi^+$, we define the \emph{Hessenberg variety} $\Hess(x,I)$ by
\begin{align*}
\Hess(x,I) = \{gB \in G/B \mid \Ad(g^{-1})(x) \in H_I \}.
\end{align*}
If $x$ is a regular nilpotent element $N$ in $\g$, then we call $\Hess(N,I)$ a \emph{regular nilpotent Hessenberg variety}.

A parabolic subgroup $P$ of $G$ containing $B$ is of the form $P_\Theta$ for some $\Theta \subset \Delta$ (see Section~\ref{sect:setting} for the definition of $P_\Theta$). 
We denote by $\p=\p_\Theta$ the Lie algebra of $P$. 
Let $\Phi_\Theta^+$ be the set of positive roots which are linear combinations of roots in $\Theta$. 
We introduce a partial order $\preceq_\Theta$ on $\Phi$ with respset to $\Theta$ in a natural way. 
A lower ideal $I \subset \Phi^+$ is a \emph{$\Theta$-ideal} if $I$ is an upper closed subset of $\Phi^+$ with respect to the partial order $\preceq_\Theta$ on $\Phi^+$ and $I$ includes $\Phi^+_\Theta$. 
Then $H_I$ is a $\p$-submodule of $\g$ containing $\p$.  
Conversely, every $\p$-submodule of $\g$ containing $\p$ is of the form $H_I$ for some $\Theta$-ideal $I \subset \Phi^+$.
See Section~\ref{subsection:Theta-ideals} for the details.
For $x \in \g$ and a $\Theta$-ideal $I \subset \Phi^+$, we define the \emph{partial Hessenberg variety} $\Hess_\Theta(x,I)$ by
\begin{align*}
\Hess_\Theta(x,I) = \{gP \in G/P \mid \Ad(g^{-1})(x) \in H_I \}.
\end{align*}
If $\Theta$ is the empty set, then $P=B$. In this case we denote $\Hess_\Theta(x,I)$ by $\Hess(x,I)$ for simplicity.
If we take a regular nilpotent element $N$ in $\g$, then $\Hess_\Theta(N,I)$ is called a \emph{regular nilpotent partial Hessenberg variety}. 
We will study a geometry and a topology of $\Hess_\Theta(N,I)$.

Tymoczko proved that $\Hess(x,I)$ is paved by affines for arbitrary $x \in \g$ in type $A$ in \cite{Tym06} and $\Hess(N,I)$ is paved by affines for classical Lie types by \cite{Tym07}. 
These results are generalized to arbitrary Lie types by Precup \cite{Pre13}.
Fresse gives an affine paving in a more general setting in \cite{Fre}, which includes the case of regular nilpotent partial Hessenberg varieties.
We give an alternative and simple proof for an affine paving of $\Hess_\Theta(N,I)$ in all Lie types by using Precup's result. 
We also give a summation formula for the Poincar\'e polynomial of $\Hess_\Theta(N,I)$ as follows.
Let $W$ be the Weyl group of $G$.
Let $W_\Theta$ be the subgroup of the Weyl group $W$ generated by $s_{\alpha}$ for $\alpha \in \Theta$ where $s_\alpha$ denotes the reflection associated to a root $\alpha$.
We write $W^\Theta$ for the minimal left coset representatives, i.e. $W^\Theta = \{w \in W \mid \ell(w) < \ell(ws_\alpha) \ \textrm{for all} \ \alpha \in \Theta \}$. 
Then the Poincar\'e polynomial of $\Hess_\Theta(N,I)$ is
\begin{align} \label{eq:summation1_Intro}
\Poin(\Hess_\Theta(N,I), \sqrt{\q}) = \sum_{w \in W^\Theta \atop w^{-1}(\Delta) \subset (-I) \cup \Phi^+} \q^{|\NN(w) \cap I|}
\end{align}
where $\NN(w)$ is the set of positive roots $\alpha$ such that $w(\alpha)$ is a negative root. 
This formula can be written as follows.
Let $\mathcal{W}^I$ be the Weyl type subsets of a lower ideal $I \subset \Phi^+$.
See Section~\ref{subsection:Weyl type subset} for the definition. 
For a $\Theta$-ideal $I \subset \Phi^+$, we set $\mathcal{W}^{I, \Theta} = \{Y \in \mathcal{W}^I \mid Y \cap \Phi^+_\Theta = \emptyset \}$. 
Then the summation formula in \eqref{eq:summation1_Intro} can be written as follows.

\begin{theorem} \label{theorem:summation2_Intro}
Let $N$ be a regular nilpotent element in $\g$ and $I$ a $\Theta$-ideal in $\Phi^+$.
Then the Poincar\'e polynomial of $\Hess_\Theta(N,I)$ is given by
\begin{align*} 
\Poin(\Hess_\Theta(N,I), \sqrt{\q}) = \sum_{Y \in \mathcal{W}^{I, \Theta}} \q^{|Y|}.
\end{align*} 
\end{theorem}

We also provide a product formula for the Poincar\'e polynomial of $\Hess_\Theta(N,I)$.
We first show that the natural projection $\pi_I: \Hess(N,I) \to \Hess_\Theta(N,I)$ is a fiber bundle with fiber $P/B$. 
Applying the Leray--Hirsch theorem, we obtain
\begin{align} \label{eq:product1_Intro}
\Poin(\Hess_\Theta(N,I),\sqrt{\q}) = \prod_{\alpha \in I \setminus \Phi^+_\Theta} \frac{1-\q^{\height(\alpha)+1}}{1-\q^{\height(\alpha)}} 
\end{align}
where $\height(\alpha)$ denotes the height of a root $\alpha$. 
We can also describe a product formula for $\Poin(\Hess_\Theta(N,I),\sqrt{\q})$ in terms of the height distribution as follows.
For a $\Theta$-ideal $I \subset \Phi^+$, we define
\begin{align*}
\lambda_i^{I \setminus \Phi^+_\Theta} = |\{\alpha \in I \setminus \Phi^+_\Theta \mid \height(\alpha)=i \}| \ \ \ \textrm{for} \ 1 \leq i \leq \rr \coloneqq \max\{\height(\alpha) \mid \alpha \in I \setminus \Phi^+_\Theta \}.
\end{align*}
Note that $I \supset \Phi^+_\Theta$ since $I$ is a $\Theta$-ideal. 
We call a sequence $(\lambda_1^{I \setminus \Phi^+_\Theta}, \lambda_2^{I \setminus \Phi^+_\Theta}, \cdots, \lambda_\rr^{I \setminus \Phi^+_\Theta})$ the \emph{height distribution in $I \setminus \Phi^+_\Theta$}. 
We set an integer 
\begin{align*} 
\mm_i^{I \setminus \Phi^+_\Theta} = \lambda_i^{I \setminus \Phi^+_\Theta}-\lambda_{i+1}^{I \setminus \Phi^+_\Theta} \ \ \ \textrm{for} \ 1 \leq i \leq \rr 
\end{align*}
with the convention $\lambda_{\rr+1}^{I \setminus \Phi^+_\Theta} = 0$.

\begin{theorem} \label{theorem:product2_Intro}
Let $N$ be a regular nilpotent element in $\g$ and $I$ a $\Theta$-ideal in $\Phi^+$. 
In the setting above, the Poincar\'e polynomial of $\Hess_\Theta(N,I)$ is equal to 
\begin{align*} 
\Poin(\Hess_\Theta(N,I), \sqrt{\q}) = \prod_{i=1}^\rr (1+\q+\q^2+\cdots+\q^i)^{\mm_i^{I \setminus \Phi^+_\Theta}}.
\end{align*} 
\end{theorem}

Note that the right hand side of the equality in Theorem~\ref{theorem:product2_Intro} seems to be a rational function, but it is a polynomial in the variable $\q$ since the left hand side is a polynomial. 
We will see an explicit polynomial description for the product formula of $\Poin(\Hess_\Theta(N,I), \sqrt{\q})$ for type $A$ in Section~\ref{subsection:typeAformula}. 
Theorems~\ref{theorem:summation2_Intro} and \ref{theorem:product2_Intro} yield the following corollary.

\begin{corollary}
For arbitrary $\Theta$-ideal $I$ in $\Phi^+$, we obtain 
\begin{align} \label{eq:Sommers-Tymoczko_Theta_Intro} 
\sum_{Y \in \mathcal{W}^{I, \Theta}} \q^{|Y|} = \prod_{i=1}^\rr (1+\q+\q^2+\cdots+\q^i)^{\mm_i^{I \setminus \Phi^+_\Theta}}. 
\end{align} 
\end{corollary}

We remark that if $\Theta$ is the empty set, then the equality in \eqref{eq:Sommers-Tymoczko_Theta_Intro} was posed as a conjecture by Sommers and Tymoczko in \cite{SoTy} and they proved this equality for types $A,B,C, F_4, E_6$, and $G_2$.
Also, Schauenburg confirmed Sommers--Tymoczko conjecture for types $D_5, D_6, D_7$, and $E_7$ by direct computation, and R\"{o}hrle proved the conjecture for type $D_4$ and $E_8$ in \cite{Roh}.
For arbitrary Lie types, Abe, Masuda, Murai, Sato with the author completely proved Sommers--Tymoczko conjecture by a classification-free argument in \cite{AHMMS}.
The equality in \eqref{eq:Sommers-Tymoczko_Theta_Intro} is a generalization of the equality. 

We next study the cohomology rings of regular nilpotent partial Hessenberg varieties. 
For this, we recall the cohomology rings of partial flag varieties. 
The flag variety $G/B$ admits an action of the Weyl group $W$, so we obtain the $W$-action on $H^*(G/B)$.  
The natural projection $\pi: G/B \to G/P$ induces the homomorphism $\pi^*: H^*(G/P) \rightarrow H^*(G/B)$.
It is known from \cite{BGG} that $\pi^*$ is injective and its image coincides with $H^*(G/B)^{W_\Theta}$ where $H^*(G/B)^{W_\Theta}$ denotes the invariants in $H^*(G/B)$ under the action of $W_\Theta$. 
In particular, $\pi^*: H^*(G/P) \hookrightarrow H^*(G/B)$ induces the isomorphism of graded $\Q$-algebras
\begin{align} \label{eq:BGG_Intro} 
H^*(G/P) \cong H^*(G/B)^{W_\Theta}.
\end{align} 
We generalize this result to regular nilpotent partial Hessenberg varieties. 

\begin{theorem} \label{theorem:cohomology_Intro}
Let $N$ be a regular nilpotent element in $\g$ and $I$ a $\Theta$-ideal in $\Phi^+$.
Then the following holds.
\begin{enumerate}
\item[(1)] The $W_\Theta$-action on $G/B$ preserves $\Hess(N,I)$.
\item[(2)] The homomorphism $\pi_I^*: H^*(\Hess_\Theta(N,I)) \rightarrow H^*(\Hess(N,I))$ induced from the natural projection $\pi_I: \Hess(N,I) \to \Hess_\Theta(N,I)$ is injective.
\item[(3)] The image of $\pi_I^*: H^*(\Hess_\Theta(N,I)) \hookrightarrow H^*(\Hess(N,I))$ coincides with the invariant subring $H^*(\Hess(N,I))^{W_\Theta}$. 
In particular, $\pi_I^*$ yields the isomorphism
\begin{align*}
H^*(\Hess_\Theta(N,I)) \cong H^*(\Hess(N,I))^{W_\Theta}
\end{align*} 
as graded $\Q$-algebras. 
\end{enumerate}
\end{theorem}

Note that if $I \subset \Phi^+$ is a $\Theta$-ideal, then the $W_\Theta$-action on $G/B$ preserves $\Hess(x,I)$ for arbitrary $x \in \g$ (not necessarily a regular nilpotent element). 
See Lemma~\ref{lemma:WTheta-action on X(x,H)} for the detail. 

There is an interesting connection between regular nilpotent Hessenberg varieties and certain hyperplane arrangements called ideal arrangements.
The connection was first studied by Sommers and Tymoczko in \cite{SoTy}, and Abe, Masuda, Murai, Sato with the author found a concrete connection in \cite{AHMMS}. 
We generalize the connection to regular nilpotent partial Hessenberg varieties.
For this purpose, we explain the work of \cite{AHMMS}.

Let $\t^*_{\Z}$ be a lattice which is identified with the character group of $T$.
We set $\t^*_{\Q}=\t^*_{\Z} \otimes_{\Z} \Q$ and its symmetric algebra $\RR = \Sym \t^*_{\Q}$. 
The $W$-action on $\t^*_{\Q}$ naturally extends a $W$-action on $\RR$. 
To each $\alpha \in \t^*_\Z$, we assign the line bundle $L_\alpha$ over the flag variety $G/B$. 
The mapping which sends $\alpha \in \t^*_\Z$ to the first Chern class $c_1(L_{\alpha}^*)$ of the dual line bundle $L_{\alpha}^*$ yields the ring homomorphism
\begin{align*} 
\varphi: \RR \rightarrow H^*(G/B); \ \ \ \alpha \mapsto c_1(L_{\alpha}^*) 
\end{align*} 
which doubles the grading on $\RR$.
The map $\varphi$ is a surjective $W$-equivariant map and its kernel is the ideal $(\RR^W_+)$ generated by the $W$-invariants in $\RR$ with zero constant term by Borel's theorem (\cite{Bor53}). 
In particular, this induces an isomorphism of graded $\Q$-algebras
\begin{align*} 
H^*(G/B) \cong \RR/(\RR^W_+).  
\end{align*} 
Let $I \subset \Phi^+$ be a lower ideal. 
We write $\varphi_I$ for the composition of $\varphi$ and the restriction map induced from the inclusion $\Hess(N,I) \subset G/B$.
By the result of \cite{AHMMS} the restriction map $H^*(G/B) \rightarrow H^*(\Hess(N,I))$ is surjective, so the map $\varphi_I$ is also surjective, i.e.  
\begin{align*} 
\varphi_I: \RR \twoheadrightarrow H^*(G/B) \twoheadrightarrow H^*(\Hess(N,I)).
\end{align*}
In \cite{AHMMS} the kernel of $\varphi_I$ is described in terms of the logarithmic derivation module of the ideal arrangement, as explained below. 
The \emph{derivation module} $\Der \RR$ of $\RR$ is the collection of all derivations of $\RR$ over $\Q$. 
For a lower ideal $I \subset \Phi^+$, we define the \emph{ideal arrangement} $\AA_I$ as the set of hyperplanes orthogonal to $\alpha \in I$.
The \emph{logarithmic derivation module} $D(\AA_I)$ of the ideal arrangement $\AA_I$ is defined by
\begin{align*}
D(\AA_I) = \{\der \in \Der \RR \mid \der(\alpha) \in \RR \, \alpha \ \textrm{for all} \ \alpha \in I \}, 
\end{align*}
which is an $\RR$-module.
Take a $W$-invariant non-degenerate quadratic form $Q \in \Sym^2 (\t^*_\Q)^W$.
Then we define an ideal $\a(I)$ of $\RR$ by 
\begin{align} \label{eq:a(I)}
\a(I) = \{\der(Q) \in \RR \mid \der \in D(\AA_I) \}, 
\end{align}
which is independent of a choice of $Q$.
By the result of \cite{AHMMS} the kernel of the surjective map $\varphi_I$ coincides with $\a(I)$. 
In particular, we have an isomorphism of graded $\Q$-algebras
\begin{align*} 
H^*(\Hess(N,I)) \cong \RR/\a(I).  
\end{align*} 
We generalize this result to regular nilpotent partial Hessenberg varieties. 

Let $I$ be a $\Theta$-ideal in $\Phi^+$. 
Considering the $W_\Theta$-invariants of the map $\varphi_I: \RR \twoheadrightarrow H^*(G/B) \twoheadrightarrow H^*(\Hess(N,I))$, we obtain $\varphi_I^{W_\Theta}: \RR^{W_\Theta} \twoheadrightarrow H^*(G/B)^{W_\Theta} \twoheadrightarrow H^*(\Hess(N,I))^{W_\Theta}$.
It follows from \eqref{eq:BGG_Intro} and Theorem~\ref{theorem:cohomology_Intro} that the map $\varphi_I^{W_\Theta}$ can be written as 
\begin{align*} 
\varphi_{I, \Theta}: \RR^{W_\Theta} \twoheadrightarrow H^*(G/P) \twoheadrightarrow H^*(\Hess_\Theta(N,I))
\end{align*}
where the first map is the composition of the isomorphism \eqref{eq:BGG_Intro} and $\varphi^{W_\Theta}: \RR^{W_\Theta} \twoheadrightarrow H^*(G/B)^{W_\Theta}$, and the second map is the restriction map induced from the inclusion $\Hess_\Theta(N,I) \subset G/P$. 
Our goal is to describe the kernel of $\varphi_{I, \Theta}$ in terms of $W_\Theta$-invariants in the logarithmic derivation module $D(\AA_I)$. 
The $W$-action on $\RR=\Sym \t^*_{\Q}$ induces a $W$-action on the derivation module $\Der \RR$. We set
\begin{align*}
(\Der \RR)^{W_\Theta} = \{\der \in \Der \RR \mid w \cdot \der = \der \ \textrm{for all} \ w \in W_\Theta \},
\end{align*}
which is an $\RR^{W_\Theta}$-module. 
We show that if $I \subset \Phi^+$ is a $\Theta$-ideal, then the $W_\Theta$-action on $\Der \RR$ preserves the logarithmic derivation module $D(\AA_I)$ (see Proposition~\ref{proposition:WTheta-action}).
For a $\Theta$-ideal $I \subset \Phi^+$, we define an $\RR^{W_\Theta}$-submodule of $(\Der \RR)^{W_\Theta}$ by 
\begin{align*}
D(\AA_I)^{W_\Theta} = D(\AA_I) \cap (\Der \RR)^{W_\Theta} = \{\der \in (\Der \RR)^{W_\Theta} \mid \der(\alpha) \in \RR \, \alpha \ \textrm{for all} \ \alpha \in I \}.
\end{align*}
Let $Q \in \Sym^2 (\t^*_\Q)^W$ be a $W$-invariant non-degenerate quadratic form.
Then we define an ideal $\a(I)_\Theta$ of $\RR^{W_\Theta}$ by
\begin{align*}
\a(I)_\Theta \coloneqq \{\der(Q) \in \RR^{W_\Theta} \mid \der \in D(\AA_I)^{W_\Theta} \}.\end{align*}
Note that $\der(Q)$ belongs to $\RR^{W_\Theta}$ for any $\der \in (\Der \RR)^{W_\Theta}$.
We also note that the ideal $\a(I)_\Theta$ does not dependend on a choice of $Q$. 

\begin{theorem}
For any $\Theta$-ideal $I \subset \Phi^+$, the kernel of the surjective map $\varphi_{I, \Theta}$ coincides with the ideal $\a(I)_\Theta$. 
In particular, $\varphi_{I,\Theta}$ induces the isomorphism
\begin{align*} 
H^*(\Hess_\Theta(N,I)) \cong \RR^{W_\Theta}/\a(I)_\Theta
\end{align*} 
as graded $\Q$-algebras.
\end{theorem}

The paper is organized as follows.
After reviewing the definitions for partial Hessenberg varieties and $\Theta$-ideals in Section~\ref{sect:setting}, we give a summation formula and a product formula for the Poincar\'e polynomial of regular nilpotent partial Hessenberg varieties in Sections~\ref{sect:summation formula} and \ref{sect:product formula}, respectively. 
In Section~\ref{sect:cohomology} we describe the cohomology rings of regular nilpotent partial Hessenberg varieties as the $W_\Theta$-invariants in the cohomology rings of regular nilpotent Hessenberg varieties.
We finally provide an expression of the cohomology rings for regular nilpotent partial Hessenberg varieties in terms of $W_\Theta$-invariants in the logarithmic derivation modules of ideal arrangements in Section~\ref{sect:invariants logarithmic derivation modules}.

\bigskip
\noindent \textbf{Acknowledgements.}  
The author is grateful to Hiraku Abe, Mikiya Masuda, Takashi Sato, and Haozhi Zeng for a valuable discussion.
The author is supported in part by JSPS KAKENHI Grant-in-Aid for Early-Career Scientists: 23K12981.

\bigskip

\section{Setting} \label{sect:setting}

In this section we define partial Hessenberg spaces and partial Hessenberg varieties. Then we introduce a notion of $\Theta$-ideals which are in one-to-one correspondence to partial Hessenberg spaces.

\subsection{Partial Hessenberg varieties}

Throughout the article, we introduce the notation below as explained in Introduction. 
\begin{enumerate}
\item[$\bullet$] $G:$ a semisimple linear algebraic group over $\C$ of rank $n$.
\item[$\bullet$] $B:$ a fixed Borel subgroup of $G$.
\item[$\bullet$] $U:$ the unipotent radical of $B$.
\item[$\bullet$] $T:$ a maximal torus of $G$ in $B$.
\item[$\bullet$] $\t \subset \b \subset \g:$ the Lie algebras of $T \subset B \subset G$ respectively.
\item[$\bullet$] $\uu:$ the Lie algebras of $U$.
\item[$\bullet$] $\t^*:$ the dual space of $\t$.
\item[$\bullet$] $\Phi \subset \t^*:$ the root system of $\t$ in $\g$.
\item[$\bullet$] $\Phi^+:$ the set of positive roots in $\Phi$, i.e. the set of roots of $\t$ in $\uu$.
\item[$\bullet$] $\Delta=\{\alpha_1,\ldots,\alpha_n\}:$ the set of simple roots.
\item[$\bullet$] $\g_\alpha:$ the root space associated to a root $\alpha \in \Phi$.
\item[$\bullet$] $E_\alpha:$ a basis of $\g_\alpha$.
\item[$\bullet$] $\height(\alpha):$ the height of a root $\alpha$, i.e. $\height(\alpha)=\sum_{i=1}^n c_i$ when we write $\alpha = \sum_{i=1}^n c_i \alpha_i$.
\item[$\bullet$] $W=N_G(T)/T:$ the Weyl group of $G$ where $N_G(T) = \{g \in G \mid gT =Tg \}$ is the normalizer of $T$.
\item[$\bullet$] $s_\alpha \in W:$ the reflection associated to a root $\alpha$.
\item[$\bullet$] $s_i \in W:$ the simple reflection associated with the simple root $\alpha_i$.
\end{enumerate}

We also frequently write
\begin{align*}
[n] \coloneqq \{1,2,\ldots,n\}.
\end{align*}
A subspace $H \subset \g$ is called a \emph{Hessenberg space} if $H$ is $\b$-submodule and $H$ contains $\b$. 
For $x \in \g$ and a Hessenberg space $H \subset \g$, the \emph{Hessenberg variety} $\Hess(x,H)$ is defined to be the following subvariety of the flag variety $G/B$:
\begin{align*}
\Hess(x,H) \coloneqq \{gB \in G/B \mid \Ad(g^{-1})(x) \in H \}.
\end{align*}
An element $x \in \g$ is \emph{nilpotent} if $\ad(X)$ is nilpotent, i.e., $\ad(X)^k=0$ for some $k > 0$.
An element $x \in \g$ is \emph{regular} if its $G$-orbit of the adjoint action has the largest possible dimension.
If $x$ is a regular nilpotent element $N$ of $\g$, then $\Hess(N,H)$ is called a \emph{regular nilpotent Hessenberg variety}. 
We note that any two regular nilpotent Hessenberg varieties are isomorphic (\cite[Lemma~5.1]{HaTy17}).

Fix a subset $\Theta \subset \Delta$ and we use the following notation throughout the manuscript.
\begin{enumerate}
\item[$\bullet$] $\Phi_\Theta^+:$ the set of positive roots which are linear combinations of roots in $\Theta$.
\item[$\bullet$] $\Levi_\Theta:$ the subgroup of $G$ generated by $T$ together with the subgroups $U_\alpha=\{\exp tE_\alpha \mid t \in \C \}$ for $\alpha \in \Phi_\Theta^+ \cup (-\Phi_\Theta^+)$ where $-\Phi^+_\Theta=\{-\alpha \in \Phi \mid \alpha \in \Phi^+_\Theta \}$.
\item[$\bullet$] $U_\Theta:$ the subgroup of $U$ generated by $U_\alpha$ for $\alpha \in \Phi^+ \setminus \Phi_\Theta^+$.
\item[$\bullet$] $P=P_\Theta \coloneqq \Levi_\Theta \ltimes U_\Theta:$ the parabolic subgroup of $G$ containing $B$.
\item[$\bullet$] $\p=\p_\Theta:$ the Lie algebra of $P$.
\item[$\bullet$] $W_\Theta:$ the subgroup of $W$ generated by $s_i=s_{\alpha_i}$ for $\alpha_i \in \Theta$.
\end{enumerate}

A subspace $H \subset \g$ is called a \emph{$\p$-Hessenberg space} (or a \emph{partial Hessenberg space}) if $H$ is $\p$-submodule and $H$ contains $\p$. 
Note that a $\p$-Hessenberg space is a (usual) Hessenberg space since $\p \supset \b$.
For $x \in \g$ and a $\p$-Hessenberg space $H$, we define the \emph{partial Hessenberg variety} $\Hess_\Theta(x,H)$ as the following subvariety of the partial flag variety $G/P$:
\begin{align*}
\Hess_\Theta(x,H) \coloneqq \{gP \in G/P \mid \Ad(g^{-1})(x) \in H \}.
\end{align*}
Note that for $x, x' \in \g$ with $x'=\Ad(g')(x)$ for some $g' \in G$, we have the isomorphism $\Hess_\Theta(x,H) \cong \Hess_\Theta(x',H)$ which sends $gP$ to $g'gP$.
If a subset $\Theta \subset \Delta$ is the empty set, then $P=B$ and we write $\Hess_\Theta(x,H)$ by $\Hess(x,H)$.
We say that $\Hess_\Theta(N,H)$ is a \emph{regular nilpotent partial Hessenberg variety} when $N$ is a regular nilpotent element of $\g$. 
Recall that we fix a basis $E_\alpha$ for each root space $g_\alpha$. Let $N_0$ be the regular nilpotent element of the form 
\begin{align} \label{eq:N0}
N_0 \coloneqq \sum_{i=1}^n E_{\alpha_i}.
\end{align}
Since we can write $N_0=\Ad(g_0)(N)$ for some $g_0 \in G$, we have
\begin{align} \label{eq:Hess(N0,H)Theta}
\Hess_\Theta(N,H) \cong \Hess_\Theta(N_0,H)
\end{align}
which sends $gP$ to $g_0gP$ (cf. \cite[Lemma~5.1]{HaTy17}).
In particular, any two regular nilpotent partial Hessenberg varieties are isomorphic.

\begin{remark}
Let $H \subset \g$ be a $\p$-Hessenberg space. 
Then the partial Hessenberg variety $\Hess_\Theta(x,H)$ is the fiber of $G \times_P H \rightarrow \g; \ [g,y] \mapsto \Ad(g)(y)$ over $x \in \g$.
Here, $G \times_P H$ denotes the quotient of the direct product $G \times H$ by the left $P$-action given by $p \cdot (g,y) = (gp^{-1},\Ad(p)(y))$ for $g \in G$, $y \in H$, and $p \in P$. 
In particular, partial Hessenberg varieties over elements in the same adjoint orbit are isomorphic. 
\end{remark}

\subsection{$\Theta$-ideals} \label{subsection:Theta-ideals}
A \emph{lower ideal} $I \subset \Phi^+$ is a collection of positive roots such that if $\alpha \in I, \beta \in \Phi^+$, and $\alpha-\beta \in \Phi^+$, then $\alpha-\beta \in I$. 
The partial order $\preceq$ on $\Phi$ is defined as follows:
\begin{align*}
\textrm{for} \ \alpha, \beta \in \Phi, \ \alpha \preceq \beta \iff \beta-\alpha \ \textrm{is a sum of positive roots.}
\end{align*}
Then, a subset $I \subset \Phi^+$ is a lower ideal if and only if it satisfies the following condition:
\begin{align} \label{eq:lower ideal}
\alpha \in I, \gamma \in \Phi^+, \gamma \preceq \alpha \Rightarrow \gamma \in I.
\end{align} 
To each lower ideal $I$ we assign a subspace $H_I \subset \g$ defined by
\begin{align} \label{eq:HI}
H_I \coloneqq \b \oplus \bigoplus_{\alpha \in I} \g_{-\alpha}.
\end{align}
One can easily see that $H_I$ is a Hessenberg space and there is a one-to-one correspondence as follows:
\begin{align*} 
\{\textrm{lower ideals in} \ \Phi^+ \} \xrightarrow{1:1} \{\textrm{Hessenberg spaces in} \ \g \}; \ I \mapsto H_I.
\end{align*}

We now explain a notion of $\Theta$-ideals which are in one-to-one correspondence to $\p$-Hessenberg spaces.
For this purpose, we introduce a partial order $\preceq_\Theta$ on $\Phi$ with respset to $\Theta$ as follows:
\begin{align*}
\textrm{for} \ \alpha, \beta \in \Phi, \ \alpha \preceq_\Theta \beta \iff \beta-\alpha \ \textrm{is a sum of positive roots in} \ \Phi^+_\Theta. 
\end{align*}
We define an \emph{upper ideal with respect to $\Theta$} as a collection of positive roots $I \subset \Phi^+$ such that if $\alpha \in I, \beta \in \Phi^+_\Theta$, and $\alpha+\beta \in \Phi^+$, then $\alpha+\beta \in I$. 
The following lemma can be proved by a similar argument of the equivalent condition \eqref{eq:lower ideal} for lower ideals.
We provide a proof for the readers' convenience.

\begin{lemma}
A subset $I \subset \Phi^+$ is an upper ideal with respect to $\Theta$ if and only if it satisfies the following condition:
\begin{align} \label{eq:Theta upper ideal}
\alpha \in I, \gamma \in \Phi^+, \alpha \preceq_\Theta \gamma \Rightarrow \gamma \in I.
\end{align} 
\end{lemma}

\begin{proof}
It is clear that if $I$ satisfies the condition \eqref{eq:Theta upper ideal}, then $I$ is an upper ideal with respect to $\Theta$.
In fact, if $\alpha \in I, \beta \in \Phi^+_\Theta$, and $\alpha+\beta \in \Phi^+$, then $\alpha \preceq_\Theta \alpha+\beta$. By \eqref{eq:Theta upper ideal} we have $\alpha+\beta \in I$. 

We show the opposite implication. Assume that $I$ is an upper ideal with respect to $\Theta$.
We take arbitrary $\gamma \in \Phi^+$. 
Then for any $\alpha \in I$ with $\alpha \preceq_\Theta \gamma$, we show that $\gamma \in I$ by a descending induction on $\height(\alpha)$. 
The base case is $\alpha=\gamma$, which is clear. 
Now assume that $\height(\alpha) < \height(\gamma)$ and that the claim holds for any $\alpha' \in I$ with $\alpha' \preceq_\Theta \gamma$ and $\height(\alpha')=\height(\alpha)+1$. 
Noting that $\alpha + (\gamma -\alpha)=\gamma \in \Phi^+$, if $\gamma - \alpha \in \Phi^+_\Theta$, then we have $\gamma \in I$ since $I$ is an upper ideal with respect to $\Theta$.
Consider the case when $\gamma - \alpha \notin \Phi^+_\Theta$.
Since $\alpha \preceq_\Theta \gamma$, one can write $\gamma-\alpha = \sum_{i \in M} c_i \alpha_i \ (c_i>0)$ for some $M \subset [n]$ with $\{\alpha_i \mid i \in M \} \subset \Theta$.
Let $( \ , \ )$ be the $W$-invariant inner product on the $\R$-span of $\Phi$ transferred from the Killing form of $\g$ restricted to $\t$. 
We claim that $(\alpha,\alpha_i)<0$ for some $i \in M$.
In fact, if we suppose that $(\alpha,\alpha_i) \geq 0$ for all $i \in M$, then $(\alpha,\gamma-\alpha)=\sum_{i \in M} c_i (\alpha,\alpha_i) \geq 0$ which implies $(\alpha, \gamma) \geq (\alpha, \alpha) > 0$.
By \cite[Lemma~9.4]{Hum72} we have $\gamma-\alpha \in \Phi^+$.
This means that $\gamma-\alpha$ is a positive root that are linear combinations of roots in $\Theta$, namely $\gamma-\alpha \in \Phi^+_\Theta$. This is a contradiction. 
Hence, we have $(\alpha,\alpha_i)<0$ for some $i \in M$.
Setting $\alpha' = \alpha+\alpha_i$, one has $\alpha' \in \Phi^+$ by \cite[Lemma~9.4]{Hum72} again. 
Since $\alpha \in I$ and $\alpha_i \in \Theta \subset \Phi^+_\Theta$, we also have $\alpha' \in I$ by the assumption that $I$ is an upper ideal with respect to $\Theta$. 
It is clear that $\alpha' \preceq_\Theta \gamma$ and $\height(\alpha')=\height(\alpha)+1$.
By our descending induction hypothesis on $\height(\alpha)$, we have $\gamma \in I$ as desired.
\end{proof}

\begin{definition}
We say that $I \subset \Phi^+$ is a \emph{$\Theta$-ideal} if $I$ satisfies the following three conditions:
\begin{enumerate}
\item[(1)] $I$ is a lower ideal;
\item[(2)] $I$ is an upper ideal with respect to $\Theta$;
\item[(3)] $I \supset \Phi^+_\Theta$. 
\end{enumerate}
\end{definition}

\begin{remark}
If $\Theta$ is the empty set, then the set of $\Theta$-ideals coincides with the set of lower ideals. 
\end{remark}

\begin{lemma} \label{lemma:one-to-one partial}
Let $\Theta \subset \Delta$ and $\p=\p_\Theta$. 
Then there is a one-to-one correspondence
\begin{align} \label{eq:one-to-one partial}
\{\Theta\textrm{-ideals in} \ \Phi^+ \} \xrightarrow{1:1} \{\p\textrm{-Hessenberg spaces in} \ \g \}
\end{align}
which sends $I$ to $H_I$ defined in \eqref{eq:HI}.
In other words, we have the following commutative diagram:
\begin{center}
\begin{tikzcd}
\{\textrm{lower ideals in} \ \Phi^+ \} \arrow[r, "1:1"] &[0.5em] \{\textrm{Hessenberg spaces in} \ \g \} \\
\{\Theta\textrm{-ideals in} \ \Phi^+ \} \arrow[r, "1:1"']  \arrow[hookrightarrow,u, "{\rm inclusion}"]                                         & \{\p\textrm{-Hessenberg spaces in} \ \g \} \arrow[hookrightarrow,u, "{\rm inclusion}"'] 
\end{tikzcd}
\end{center}
\end{lemma}

\begin{proof}
Let $I$ be a $\Theta$-ideal. 
We first show that $H_I = \b \oplus \bigoplus_{\alpha \in I} \g_{-\alpha}$ is a $\p$-Hessenberg space. 
Since $\p=\p_\Theta=\b \oplus \bigoplus_{\beta \in \Phi^+_\Theta} \g_{-\beta}$ and $I \supset \Phi^+_\Theta$, one has $H_I \supset \p$. 
In order to show that $H_I$ is a $\p$-submodule, it suffices to show that $[\g_{-\beta}, H_I] \subset H_I$ for all $\beta \in \Phi^+_\Theta$ since we already know that $H_I$ is a $\b$-submodule by the condition that $I$ is a lower ideal.
We see that $[\g_{-\beta},\b]=-[\b,\g_{-\beta}] \subset -[\b,H_I] \subset H_I$ for any $\beta \in \Phi^+_\Theta$ since $H_I$ is a $\b$-submodule, so we need to check that $[\g_{-\beta},\g_{-\alpha}] \subset H_I$ for all $\alpha \in I$ and $\beta \in \Phi^+_\Theta$ with $\alpha+\beta \in \Phi^+$. 
By the condition that $I$ is an upper ideal with respect to $\Theta$, we have $\alpha+\beta \in I$, which implies $[\g_{-\beta},\g_{-\alpha}] = \g_{-(\alpha+\beta)} \subset H_I$. 
Hence, $H_I$ is a $\p$-Hessenberg space. 

The map $I \mapsto H_I$ is clearly injective, so we prove the surjectivity for the map. 
Let $H$ be a $\p$-Hessenberg space. Since $H$ contains $\p=\b \oplus \bigoplus_{\beta \in \Phi^+_\Theta} \g_{-\beta}$, we can write $H=H_I$ for some unique lower ideal $I \supset \Phi^+_\Theta$.  
We prove that $I$ is an upper ideal with respect to $\Theta$.
If $\alpha \in I$ and $\beta \in \Phi^+_\Theta$ with $\alpha+\beta \in \Phi^+$, then we have $\g_{-(\alpha+\beta)} = [\g_{-\beta},\g_{-\alpha}] \subset [\p,H_I] \subset H_I$ since $H=H_I$ is a $\p$-submodule.
This implies that $\alpha+\beta \in I$, so $I$ is an upper ideal with respect to $\Theta$ and hence $I$ is a $\Theta$-ideal as desired. 
\end{proof}

Fix a regular nilpotent element $N \in \g$.
By Lemma~\ref{lemma:one-to-one partial}, every regular nilpotent partial Hessenberg variety can be written as $\Hess_\Theta(N,H_I)$ for some unique $\Theta$-ideal $I$. 
From next section we will write $\Hess_\Theta(N,I)\coloneqq\Hess_\Theta(N,H_I)$ for simplicity, i.e. 
\begin{align} \label{eq:Hess(N,I)Theta}
\Hess_\Theta(N,I) = \{gP \in G/P \mid \Ad(g^{-1})(N) \in H_I \}.
\end{align}
In particular, if $\Theta$ is the empty set, i.e. $P=B$, then we denote $\Hess_\Theta(N,I)$ by $\Hess(N,I)$, namely
\begin{align} \label{eq:Hess(N,I)}
\Hess(N,I) = \{gP \in G/B \mid \Ad(g^{-1})(N) \in H_I \}.
\end{align}

We record some properties for regular nilpotent Hessenberg varieties as follows.

\begin{proposition} \label{proposition:propertiesHess(N,I)} $($\cite{AFZ, Pre13, Pre18}$)$
Let $N$ be a regular nilpotent element in $\g$ and $I$ a lower ideal.
Then the associated regular nilpotent Hessenberg variety $\Hess(N,I)$ is irreducible and its complex dimension is given by $\dim_\C \Hess(N,I) = |I|$. 
\end{proposition}

\begin{remark}
The result above for type $A$ was proved by \cite[Lemma~7.1]{AnTy} and \cite[Theorem~10.2]{SoTy}.
\end{remark}

We will see a partial Hessenberg analogue of the properties in Section~\ref{sect:product formula}.

\subsection{Type $A$ description} \label{subsection:Type A description}

We here explain partial Hessenberg varieties for type $A$ in terms of partial Hessenberg functions, which were introduced by Kiem and Lee in \cite{KiLe}. 
Let $\KK \coloneqq \{\kk_1,\ldots,\kk_s \}$ be a subset of $[n-1]$ with $1 \leq \kk_1 < \cdots < \kk_s \leq n-1$. 
The partial flag variety in type $A_{n-1}$ is defined as 
\begin{align*}
\Flag_\KK(\C^n) \coloneqq \{ (V_{\kk_p})_{1 \leq p \leq s} \coloneqq (V_{\kk_1} \subset V_{\kk_2} \subset \cdots \subset V_{\kk_s} \subset \C^n) \mid \dim V_{\kk_p} = \kk_p \ \textrm{for all} \ 1 \leq p \leq s \}.
\end{align*}
Let $G=\SL_n(\C)$ and $B$ the set of the upper triangular matrices of $G$. 
Set $\Phi^+=\{x_i-x_j \mid 1 \leq i < j \leq n\}$ and $\Delta=\{\alpha_i\coloneqq x_i-x_{i+1} \mid i \in [n-1] \}$.
Let $\Theta=\{ \alpha_i \in \Delta \mid i \in [n-1]\setminus \KK \}$ and $P=P_\Theta=\{(a_{ij})_{i,j \in [n]} \in G \mid a_{ij}=0 \ \textrm{if} \ i>\kk_p \ \textrm{and} \ j \leq \kk_p \ \textrm{for all} \ 1 \leq p \leq s \}$.
As is well-known, the partial flag variety $\Flag_\KK(\C^n)$ is naturally identified with $G/P$.

A \emph{partial Hessenberg function} is defined to be a function $h: \KK \cup \{n \} \rightarrow \KK \cup \{n \}$ such that $h(j) \geq j$ for all $j \in \KK \cup \{n \}$ and $h(\kk_1) \leq h(\kk_2) \leq \dots \leq h(\kk_s)$. 
We frequently write a partial Hessenberg function by listing its values in sequence, namely $h=(h(\kk_1), h(\kk_2), \ldots,  h(\kk_s),h(n)=n)$.
When $\KK = [n-1]$, such a function $h:[n] \rightarrow [n]$ is called a \emph{Hessenberg function}. 
Note that partial Hessenberg functions (resp. partial Hessenberg varieties) are called generalized Hessenberg functions (resp. generalized Hessenberg varieties) in \cite{KiLe}.
We can extend a partial Hessenberg function $h: \KK \cup \{n \} \rightarrow \KK \cup \{n \}$ to the Hessenberg function $\tilde{h}: [n] \to [n]$ defined by 
\begin{align} \label{eq:tilde h}
\tilde{h}(j) = h(\kk_p) \ \textrm{if} \ \kk_{p-1} < j \leq \kk_p \ \textrm{for some} \ 1 \leq p \leq s+1 
\end{align}
with the convention that $\kk_0=0$ and $\kk_{s+1}=n$. 
Remark that the extended Hessenberg functions $\tilde{h}$ are introduced in \cite{KiLe}. 
Then, one can easily see the following one-to-one correspondence  
\begin{align*}
\{\textrm{partial Hessenberg functions} \ \KK \cup \{n \} \rightarrow \KK \cup \{n \} \} \xrightarrow{1:1} \{\p\textrm{-Hessenberg spaces in} \ \g \},
\end{align*}
which sends a partial Hessenberg function $h:\KK \cup \{n \} \rightarrow \KK \cup \{n \}$ to the $\p$-Hessenberg space defined by
\begin{align*}
H(h) \coloneqq \{ (a_{ij})_{i,j \in [n]} \in \g \mid a_{ij}=0 \ \textrm{if} \ i > \tilde{h}(j) \ \textrm{for all} \ j \in [n] \}.
\end{align*}
Let $N$ be a regular nilpotent matrix (i.e. a nilpotent matrix whose Jordan form consists of exactly one Jordan block) and $h: \KK \cup \{n \} \rightarrow \KK \cup \{n \}$ a partial Hessenberg function.
Then under the identification $\Flag_\KK(\C^n) \cong G/P$, the following subvariety 
\begin{align} \label{eq:Hess(N,h)KKTypeA}
\Hess_\KK(N,h) \coloneqq \{ (V_{\kk_p})_{1 \leq p \leq s} \in \Flag_\KK(\C^n) \mid NV_{\kk_p} \subset V_{h(\kk_p)} \ \textrm{for all} \ 1 \leq p \leq s \}
\end{align}
is isomorphic to the regular nilpotent partial Hessenberg variety $\Hess_\Theta(N,H(h))$.  
It is also useful to express arbitrary Hessenberg space $H_I$ pictorially by drawing a configuration of boxes on a square grid of size $n \times n$ whose shaded boxes correspond to the roots in $\Phi^+ \cup (-I)$ appearing in \eqref{eq:HI} together with diagonal parts (see Example~\ref{example:typeAHessenberg} below).

\begin{example} \label{example:typeAHessenberg} \label{example:h=(4,5,5,8,10)}
Consider type $A_9$ and $\KK=\{2,4,5,8\}$.
Note that $\Theta=\Delta \setminus \{\alpha_2,\alpha_4,\alpha_5,\alpha_8 \}$. 
For example, $h=(4,5,5,8,10)$ is a partial Hessenberg function and the extended Hessenberg function is given by $\tilde{h}=(4,4,5,5,5,8,8,8,10,10)$. 
In this case, we have
\begin{align*}
\Phi^+_\Theta &= \{x_1-x_2, x_3-x_4, x_6-x_7, x_6-x_8, x_7-x_8, x_9-x_{10} \} \\
I &= \Phi^+_\Theta \cup \{x_1-x_3, x_1-x_4, x_2-x_3, x_2-x_4, x_3-x_5, x_4-x_5 \}
\end{align*}
where $I$ is the $\Theta$-ideal corresponding to $H(h)$ under \eqref{eq:one-to-one partial}, i.e. $H(h)=\b \oplus \bigoplus_{\alpha \in I} \g_{-\alpha}$. 
The pictures for $H(h)$ and $\p_\Theta=\b \oplus \bigoplus_{\beta \in \Phi^+_\Theta} \g_{-\beta}$ are shown in Figure~\ref{pic:h=(4,5,8,8,10)}.
\begin{figure}[h]
\begin{center}
\begin{picture}(250,100)

\put(-40,47){$H(h):$}

\put(0,93){\colorbox{gray}}
\put(0,97){\colorbox{gray}}
\put(4,93){\colorbox{gray}}
\put(4,97){\colorbox{gray}}
\put(0,83){\colorbox{gray}}
\put(0,87){\colorbox{gray}}
\put(4,83){\colorbox{gray}}
\put(4,87){\colorbox{gray}}
\put(0,73){\colorbox{gray}}
\put(0,77){\colorbox{gray}}
\put(4,73){\colorbox{gray}}
\put(4,77){\colorbox{gray}}
\put(0,63){\colorbox{gray}}
\put(0,67){\colorbox{gray}}
\put(4,63){\colorbox{gray}}
\put(4,67){\colorbox{gray}}

\put(10,93){\colorbox{gray}}
\put(10,97){\colorbox{gray}}
\put(14,93){\colorbox{gray}}
\put(14,97){\colorbox{gray}}
\put(10,83){\colorbox{gray}}
\put(10,87){\colorbox{gray}}
\put(14,83){\colorbox{gray}}
\put(14,87){\colorbox{gray}}
\put(10,73){\colorbox{gray}}
\put(10,77){\colorbox{gray}}
\put(14,73){\colorbox{gray}}
\put(14,77){\colorbox{gray}}
\put(10,63){\colorbox{gray}}
\put(10,67){\colorbox{gray}}
\put(14,63){\colorbox{gray}}
\put(14,67){\colorbox{gray}}

\put(20,93){\colorbox{gray}}
\put(20,97){\colorbox{gray}}
\put(24,93){\colorbox{gray}}
\put(24,97){\colorbox{gray}}
\put(20,83){\colorbox{gray}}
\put(20,87){\colorbox{gray}}
\put(24,83){\colorbox{gray}}
\put(24,87){\colorbox{gray}}
\put(20,73){\colorbox{gray}}
\put(20,77){\colorbox{gray}}
\put(24,73){\colorbox{gray}}
\put(24,77){\colorbox{gray}}
\put(20,63){\colorbox{gray}}
\put(20,67){\colorbox{gray}}
\put(24,63){\colorbox{gray}}
\put(24,67){\colorbox{gray}}
\put(20,53){\colorbox{gray}}
\put(20,57){\colorbox{gray}}
\put(24,53){\colorbox{gray}}
\put(24,57){\colorbox{gray}}

\put(30,93){\colorbox{gray}}
\put(30,97){\colorbox{gray}}
\put(34,93){\colorbox{gray}}
\put(34,97){\colorbox{gray}}
\put(30,83){\colorbox{gray}}
\put(30,87){\colorbox{gray}}
\put(34,83){\colorbox{gray}}
\put(34,87){\colorbox{gray}}
\put(30,73){\colorbox{gray}}
\put(30,77){\colorbox{gray}}
\put(34,73){\colorbox{gray}}
\put(34,77){\colorbox{gray}}
\put(30,63){\colorbox{gray}}
\put(30,67){\colorbox{gray}}
\put(34,63){\colorbox{gray}}
\put(34,67){\colorbox{gray}}
\put(30,53){\colorbox{gray}}
\put(30,57){\colorbox{gray}}
\put(34,53){\colorbox{gray}}
\put(34,57){\colorbox{gray}}

\put(40,93){\colorbox{gray}}
\put(40,97){\colorbox{gray}}
\put(44,93){\colorbox{gray}}
\put(44,97){\colorbox{gray}}
\put(40,83){\colorbox{gray}}
\put(40,87){\colorbox{gray}}
\put(44,83){\colorbox{gray}}
\put(44,87){\colorbox{gray}}
\put(40,73){\colorbox{gray}}
\put(40,77){\colorbox{gray}}
\put(44,73){\colorbox{gray}}
\put(44,77){\colorbox{gray}}
\put(40,63){\colorbox{gray}}
\put(40,67){\colorbox{gray}}
\put(44,63){\colorbox{gray}}
\put(44,67){\colorbox{gray}}
\put(40,53){\colorbox{gray}}
\put(40,57){\colorbox{gray}}
\put(44,53){\colorbox{gray}}
\put(44,57){\colorbox{gray}}

\put(50,93){\colorbox{gray}}
\put(50,97){\colorbox{gray}}
\put(54,93){\colorbox{gray}}
\put(54,97){\colorbox{gray}}
\put(50,83){\colorbox{gray}}
\put(50,87){\colorbox{gray}}
\put(54,83){\colorbox{gray}}
\put(54,87){\colorbox{gray}}
\put(50,73){\colorbox{gray}}
\put(50,77){\colorbox{gray}}
\put(54,73){\colorbox{gray}}
\put(54,77){\colorbox{gray}}
\put(50,63){\colorbox{gray}}
\put(50,67){\colorbox{gray}}
\put(54,63){\colorbox{gray}}
\put(54,67){\colorbox{gray}}
\put(50,53){\colorbox{gray}}
\put(50,57){\colorbox{gray}}
\put(54,53){\colorbox{gray}}
\put(54,57){\colorbox{gray}}
\put(50,43){\colorbox{gray}}
\put(50,47){\colorbox{gray}}
\put(54,43){\colorbox{gray}}
\put(54,47){\colorbox{gray}}
\put(50,33){\colorbox{gray}}
\put(50,37){\colorbox{gray}}
\put(54,33){\colorbox{gray}}
\put(54,37){\colorbox{gray}}
\put(50,23){\colorbox{gray}}
\put(50,27){\colorbox{gray}}
\put(54,23){\colorbox{gray}}
\put(54,27){\colorbox{gray}}

\put(60,93){\colorbox{gray}}
\put(60,97){\colorbox{gray}}
\put(64,93){\colorbox{gray}}
\put(64,97){\colorbox{gray}}
\put(60,83){\colorbox{gray}}
\put(60,87){\colorbox{gray}}
\put(64,83){\colorbox{gray}}
\put(64,87){\colorbox{gray}}
\put(60,73){\colorbox{gray}}
\put(60,77){\colorbox{gray}}
\put(64,73){\colorbox{gray}}
\put(64,77){\colorbox{gray}}
\put(60,63){\colorbox{gray}}
\put(60,67){\colorbox{gray}}
\put(64,63){\colorbox{gray}}
\put(64,67){\colorbox{gray}}
\put(60,53){\colorbox{gray}}
\put(60,57){\colorbox{gray}}
\put(64,53){\colorbox{gray}}
\put(64,57){\colorbox{gray}}
\put(60,43){\colorbox{gray}}
\put(60,47){\colorbox{gray}}
\put(64,43){\colorbox{gray}}
\put(64,47){\colorbox{gray}}
\put(60,33){\colorbox{gray}}
\put(60,37){\colorbox{gray}}
\put(64,33){\colorbox{gray}}
\put(64,37){\colorbox{gray}}
\put(60,23){\colorbox{gray}}
\put(60,27){\colorbox{gray}}
\put(64,23){\colorbox{gray}}
\put(64,27){\colorbox{gray}}

\put(70,93){\colorbox{gray}}
\put(70,97){\colorbox{gray}}
\put(74,93){\colorbox{gray}}
\put(74,97){\colorbox{gray}}
\put(70,83){\colorbox{gray}}
\put(70,87){\colorbox{gray}}
\put(74,83){\colorbox{gray}}
\put(74,87){\colorbox{gray}}
\put(70,73){\colorbox{gray}}
\put(70,77){\colorbox{gray}}
\put(74,73){\colorbox{gray}}
\put(74,77){\colorbox{gray}}
\put(70,63){\colorbox{gray}}
\put(70,67){\colorbox{gray}}
\put(74,63){\colorbox{gray}}
\put(74,67){\colorbox{gray}}
\put(70,53){\colorbox{gray}}
\put(70,57){\colorbox{gray}}
\put(74,53){\colorbox{gray}}
\put(74,57){\colorbox{gray}}
\put(70,43){\colorbox{gray}}
\put(70,47){\colorbox{gray}}
\put(74,43){\colorbox{gray}}
\put(74,47){\colorbox{gray}}
\put(70,33){\colorbox{gray}}
\put(70,37){\colorbox{gray}}
\put(74,33){\colorbox{gray}}
\put(74,37){\colorbox{gray}}
\put(70,23){\colorbox{gray}}
\put(70,27){\colorbox{gray}}
\put(74,23){\colorbox{gray}}
\put(74,27){\colorbox{gray}}

\put(80,93){\colorbox{gray}}
\put(80,97){\colorbox{gray}}
\put(84,93){\colorbox{gray}}
\put(84,97){\colorbox{gray}}
\put(80,83){\colorbox{gray}}
\put(80,87){\colorbox{gray}}
\put(84,83){\colorbox{gray}}
\put(84,87){\colorbox{gray}}
\put(80,73){\colorbox{gray}}
\put(80,77){\colorbox{gray}}
\put(84,73){\colorbox{gray}}
\put(84,77){\colorbox{gray}}
\put(80,63){\colorbox{gray}}
\put(80,67){\colorbox{gray}}
\put(84,63){\colorbox{gray}}
\put(84,67){\colorbox{gray}}
\put(80,53){\colorbox{gray}}
\put(80,57){\colorbox{gray}}
\put(84,53){\colorbox{gray}}
\put(84,57){\colorbox{gray}}
\put(80,43){\colorbox{gray}}
\put(80,47){\colorbox{gray}}
\put(84,43){\colorbox{gray}}
\put(84,47){\colorbox{gray}}
\put(80,33){\colorbox{gray}}
\put(80,37){\colorbox{gray}}
\put(84,33){\colorbox{gray}}
\put(84,37){\colorbox{gray}}
\put(80,23){\colorbox{gray}}
\put(80,27){\colorbox{gray}}
\put(84,23){\colorbox{gray}}
\put(84,27){\colorbox{gray}}
\put(80,13){\colorbox{gray}}
\put(80,17){\colorbox{gray}}
\put(84,13){\colorbox{gray}}
\put(84,17){\colorbox{gray}}
\put(80,3){\colorbox{gray}}
\put(80,7){\colorbox{gray}}
\put(84,3){\colorbox{gray}}
\put(84,7){\colorbox{gray}}

\put(90,93){\colorbox{gray}}
\put(90,97){\colorbox{gray}}
\put(94,93){\colorbox{gray}}
\put(94,97){\colorbox{gray}}
\put(90,83){\colorbox{gray}}
\put(90,87){\colorbox{gray}}
\put(94,83){\colorbox{gray}}
\put(94,87){\colorbox{gray}}
\put(90,73){\colorbox{gray}}
\put(90,77){\colorbox{gray}}
\put(94,73){\colorbox{gray}}
\put(94,77){\colorbox{gray}}
\put(90,63){\colorbox{gray}}
\put(90,67){\colorbox{gray}}
\put(94,63){\colorbox{gray}}
\put(94,67){\colorbox{gray}}
\put(90,53){\colorbox{gray}}
\put(90,57){\colorbox{gray}}
\put(94,53){\colorbox{gray}}
\put(94,57){\colorbox{gray}}
\put(90,43){\colorbox{gray}}
\put(90,47){\colorbox{gray}}
\put(94,43){\colorbox{gray}}
\put(94,47){\colorbox{gray}}
\put(90,33){\colorbox{gray}}
\put(90,37){\colorbox{gray}}
\put(94,33){\colorbox{gray}}
\put(94,37){\colorbox{gray}}
\put(90,23){\colorbox{gray}}
\put(90,27){\colorbox{gray}}
\put(94,23){\colorbox{gray}}
\put(94,27){\colorbox{gray}}
\put(90,13){\colorbox{gray}}
\put(90,17){\colorbox{gray}}
\put(94,13){\colorbox{gray}}
\put(94,17){\colorbox{gray}}
\put(90,3){\colorbox{gray}}
\put(90,7){\colorbox{gray}}
\put(94,3){\colorbox{gray}}
\put(94,7){\colorbox{gray}}

\put(0,0){\framebox(10,10)}
\put(10,0){\framebox(10,10)}
\put(20,0){\framebox(10,10)}
\put(30,0){\framebox(10,10)}
\put(40,0){\framebox(10,10)}
\put(50,0){\framebox(10,10)}
\put(60,0){\framebox(10,10)}
\put(70,0){\framebox(10,10)}
\put(80,0){\framebox(10,10)}
\put(90,0){\framebox(10,10)}
\put(0,10){\framebox(10,10)}
\put(10,10){\framebox(10,10)}
\put(20,10){\framebox(10,10)}
\put(30,10){\framebox(10,10)}
\put(40,10){\framebox(10,10)}
\put(50,10){\framebox(10,10)}
\put(60,10){\framebox(10,10)}
\put(70,10){\framebox(10,10)}
\put(80,10){\framebox(10,10)}
\put(90,10){\framebox(10,10)}
\put(0,20){\framebox(10,10)}
\put(10,20){\framebox(10,10)}
\put(20,20){\framebox(10,10)}
\put(30,20){\framebox(10,10)}
\put(40,20){\framebox(10,10)}
\put(50,20){\framebox(10,10)}
\put(60,20){\framebox(10,10)}
\put(70,20){\framebox(10,10)}
\put(80,20){\framebox(10,10)}
\put(90,20){\framebox(10,10)}
\put(0,30){\framebox(10,10)}
\put(10,30){\framebox(10,10)}
\put(20,30){\framebox(10,10)}
\put(30,30){\framebox(10,10)}
\put(40,30){\framebox(10,10)}
\put(50,30){\framebox(10,10)}
\put(60,30){\framebox(10,10)}
\put(70,30){\framebox(10,10)}
\put(80,30){\framebox(10,10)}
\put(90,30){\framebox(10,10)}
\put(0,40){\framebox(10,10)}
\put(10,40){\framebox(10,10)}
\put(20,40){\framebox(10,10)}
\put(30,40){\framebox(10,10)}
\put(40,40){\framebox(10,10)}
\put(50,40){\framebox(10,10)}
\put(60,40){\framebox(10,10)}
\put(70,40){\framebox(10,10)}
\put(80,40){\framebox(10,10)}
\put(90,40){\framebox(10,10)}
\put(0,50){\framebox(10,10)}
\put(10,50){\framebox(10,10)}
\put(20,50){\framebox(10,10)}
\put(30,50){\framebox(10,10)}
\put(40,50){\framebox(10,10)}
\put(50,50){\framebox(10,10)}
\put(60,50){\framebox(10,10)}
\put(70,50){\framebox(10,10)}
\put(80,50){\framebox(10,10)}
\put(90,50){\framebox(10,10)}
\put(0,60){\framebox(10,10)}
\put(10,60){\framebox(10,10)}
\put(20,60){\framebox(10,10)}
\put(30,60){\framebox(10,10)}
\put(40,60){\framebox(10,10)}
\put(50,60){\framebox(10,10)}
\put(60,60){\framebox(10,10)}
\put(70,60){\framebox(10,10)}
\put(80,60){\framebox(10,10)}
\put(90,60){\framebox(10,10)}
\put(0,70){\framebox(10,10)}
\put(10,70){\framebox(10,10)}
\put(20,70){\framebox(10,10)}
\put(30,70){\framebox(10,10)}
\put(40,70){\framebox(10,10)}
\put(50,70){\framebox(10,10)}
\put(60,70){\framebox(10,10)}
\put(70,70){\framebox(10,10)}
\put(80,70){\framebox(10,10)}
\put(90,70){\framebox(10,10)}
\put(0,80){\framebox(10,10)}
\put(10,80){\framebox(10,10)}
\put(20,80){\framebox(10,10)}
\put(30,80){\framebox(10,10)}
\put(40,80){\framebox(10,10)}
\put(50,80){\framebox(10,10)}
\put(60,80){\framebox(10,10)}
\put(70,80){\framebox(10,10)}
\put(80,80){\framebox(10,10)}
\put(90,80){\framebox(10,10)}
\put(0,90){\framebox(10,10)}
\put(10,90){\framebox(10,10)}
\put(20,90){\framebox(10,10)}
\put(30,90){\framebox(10,10)}
\put(40,90){\framebox(10,10)}
\put(50,90){\framebox(10,10)}
\put(60,90){\framebox(10,10)}
\put(70,90){\framebox(10,10)}
\put(80,90){\framebox(10,10)}
\put(90,90){\framebox(10,10)}

\hspace{-30pt}

\put(175,47){$\p_\Theta:$}

\put(200,93){\colorbox{gray}}
\put(200,97){\colorbox{gray}}
\put(204,93){\colorbox{gray}}
\put(204,97){\colorbox{gray}}
\put(200,83){\colorbox{gray}}
\put(200,87){\colorbox{gray}}
\put(204,83){\colorbox{gray}}
\put(204,87){\colorbox{gray}}

\put(210,93){\colorbox{gray}}
\put(210,97){\colorbox{gray}}
\put(214,93){\colorbox{gray}}
\put(214,97){\colorbox{gray}}
\put(210,83){\colorbox{gray}}
\put(210,87){\colorbox{gray}}
\put(214,83){\colorbox{gray}}
\put(214,87){\colorbox{gray}}

\put(220,93){\colorbox{gray}}
\put(220,97){\colorbox{gray}}
\put(224,93){\colorbox{gray}}
\put(224,97){\colorbox{gray}}
\put(220,83){\colorbox{gray}}
\put(220,87){\colorbox{gray}}
\put(224,83){\colorbox{gray}}
\put(224,87){\colorbox{gray}}
\put(220,73){\colorbox{gray}}
\put(220,77){\colorbox{gray}}
\put(224,73){\colorbox{gray}}
\put(224,77){\colorbox{gray}}
\put(220,63){\colorbox{gray}}
\put(220,67){\colorbox{gray}}
\put(224,63){\colorbox{gray}}
\put(224,67){\colorbox{gray}}

\put(230,93){\colorbox{gray}}
\put(230,97){\colorbox{gray}}
\put(234,93){\colorbox{gray}}
\put(234,97){\colorbox{gray}}
\put(230,83){\colorbox{gray}}
\put(230,87){\colorbox{gray}}
\put(234,83){\colorbox{gray}}
\put(234,87){\colorbox{gray}}
\put(230,73){\colorbox{gray}}
\put(230,77){\colorbox{gray}}
\put(234,73){\colorbox{gray}}
\put(234,77){\colorbox{gray}}
\put(230,63){\colorbox{gray}}
\put(230,67){\colorbox{gray}}
\put(234,63){\colorbox{gray}}
\put(234,67){\colorbox{gray}}

\put(240,93){\colorbox{gray}}
\put(240,97){\colorbox{gray}}
\put(244,93){\colorbox{gray}}
\put(244,97){\colorbox{gray}}
\put(240,83){\colorbox{gray}}
\put(240,87){\colorbox{gray}}
\put(244,83){\colorbox{gray}}
\put(244,87){\colorbox{gray}}
\put(240,73){\colorbox{gray}}
\put(240,77){\colorbox{gray}}
\put(244,73){\colorbox{gray}}
\put(244,77){\colorbox{gray}}
\put(240,63){\colorbox{gray}}
\put(240,67){\colorbox{gray}}
\put(244,63){\colorbox{gray}}
\put(244,67){\colorbox{gray}}
\put(240,53){\colorbox{gray}}
\put(240,57){\colorbox{gray}}
\put(244,53){\colorbox{gray}}
\put(244,57){\colorbox{gray}}

\put(250,93){\colorbox{gray}}
\put(250,97){\colorbox{gray}}
\put(254,93){\colorbox{gray}}
\put(254,97){\colorbox{gray}}
\put(250,83){\colorbox{gray}}
\put(250,87){\colorbox{gray}}
\put(254,83){\colorbox{gray}}
\put(254,87){\colorbox{gray}}
\put(250,73){\colorbox{gray}}
\put(250,77){\colorbox{gray}}
\put(254,73){\colorbox{gray}}
\put(254,77){\colorbox{gray}}
\put(250,63){\colorbox{gray}}
\put(250,67){\colorbox{gray}}
\put(254,63){\colorbox{gray}}
\put(254,67){\colorbox{gray}}
\put(250,53){\colorbox{gray}}
\put(250,57){\colorbox{gray}}
\put(254,53){\colorbox{gray}}
\put(254,57){\colorbox{gray}}
\put(250,43){\colorbox{gray}}
\put(250,47){\colorbox{gray}}
\put(254,43){\colorbox{gray}}
\put(254,47){\colorbox{gray}}
\put(250,33){\colorbox{gray}}
\put(250,37){\colorbox{gray}}
\put(254,33){\colorbox{gray}}
\put(254,37){\colorbox{gray}}
\put(250,23){\colorbox{gray}}
\put(250,27){\colorbox{gray}}
\put(254,23){\colorbox{gray}}
\put(254,27){\colorbox{gray}}

\put(260,93){\colorbox{gray}}
\put(260,97){\colorbox{gray}}
\put(264,93){\colorbox{gray}}
\put(264,97){\colorbox{gray}}
\put(260,83){\colorbox{gray}}
\put(260,87){\colorbox{gray}}
\put(264,83){\colorbox{gray}}
\put(264,87){\colorbox{gray}}
\put(260,73){\colorbox{gray}}
\put(260,77){\colorbox{gray}}
\put(264,73){\colorbox{gray}}
\put(264,77){\colorbox{gray}}
\put(260,63){\colorbox{gray}}
\put(260,67){\colorbox{gray}}
\put(264,63){\colorbox{gray}}
\put(264,67){\colorbox{gray}}
\put(260,53){\colorbox{gray}}
\put(260,57){\colorbox{gray}}
\put(264,53){\colorbox{gray}}
\put(264,57){\colorbox{gray}}
\put(260,43){\colorbox{gray}}
\put(260,47){\colorbox{gray}}
\put(264,43){\colorbox{gray}}
\put(264,47){\colorbox{gray}}
\put(260,33){\colorbox{gray}}
\put(260,37){\colorbox{gray}}
\put(264,33){\colorbox{gray}}
\put(264,37){\colorbox{gray}}
\put(260,23){\colorbox{gray}}
\put(260,27){\colorbox{gray}}
\put(264,23){\colorbox{gray}}
\put(264,27){\colorbox{gray}}

\put(270,93){\colorbox{gray}}
\put(270,97){\colorbox{gray}}
\put(274,93){\colorbox{gray}}
\put(274,97){\colorbox{gray}}
\put(270,83){\colorbox{gray}}
\put(270,87){\colorbox{gray}}
\put(274,83){\colorbox{gray}}
\put(274,87){\colorbox{gray}}
\put(270,73){\colorbox{gray}}
\put(270,77){\colorbox{gray}}
\put(274,73){\colorbox{gray}}
\put(274,77){\colorbox{gray}}
\put(270,63){\colorbox{gray}}
\put(270,67){\colorbox{gray}}
\put(274,63){\colorbox{gray}}
\put(274,67){\colorbox{gray}}
\put(270,53){\colorbox{gray}}
\put(270,57){\colorbox{gray}}
\put(274,53){\colorbox{gray}}
\put(274,57){\colorbox{gray}}
\put(270,43){\colorbox{gray}}
\put(270,47){\colorbox{gray}}
\put(274,43){\colorbox{gray}}
\put(274,47){\colorbox{gray}}
\put(270,33){\colorbox{gray}}
\put(270,37){\colorbox{gray}}
\put(274,33){\colorbox{gray}}
\put(274,37){\colorbox{gray}}
\put(270,23){\colorbox{gray}}
\put(270,27){\colorbox{gray}}
\put(274,23){\colorbox{gray}}
\put(274,27){\colorbox{gray}}

\put(280,93){\colorbox{gray}}
\put(280,97){\colorbox{gray}}
\put(284,93){\colorbox{gray}}
\put(284,97){\colorbox{gray}}
\put(280,83){\colorbox{gray}}
\put(280,87){\colorbox{gray}}
\put(284,83){\colorbox{gray}}
\put(284,87){\colorbox{gray}}
\put(280,73){\colorbox{gray}}
\put(280,77){\colorbox{gray}}
\put(284,73){\colorbox{gray}}
\put(284,77){\colorbox{gray}}
\put(280,63){\colorbox{gray}}
\put(280,67){\colorbox{gray}}
\put(284,63){\colorbox{gray}}
\put(284,67){\colorbox{gray}}
\put(280,53){\colorbox{gray}}
\put(280,57){\colorbox{gray}}
\put(284,53){\colorbox{gray}}
\put(284,57){\colorbox{gray}}
\put(280,43){\colorbox{gray}}
\put(280,47){\colorbox{gray}}
\put(284,43){\colorbox{gray}}
\put(284,47){\colorbox{gray}}
\put(280,33){\colorbox{gray}}
\put(280,37){\colorbox{gray}}
\put(284,33){\colorbox{gray}}
\put(284,37){\colorbox{gray}}
\put(280,23){\colorbox{gray}}
\put(280,27){\colorbox{gray}}
\put(284,23){\colorbox{gray}}
\put(284,27){\colorbox{gray}}
\put(280,13){\colorbox{gray}}
\put(280,17){\colorbox{gray}}
\put(284,13){\colorbox{gray}}
\put(284,17){\colorbox{gray}}
\put(280,3){\colorbox{gray}}
\put(280,7){\colorbox{gray}}
\put(284,3){\colorbox{gray}}
\put(284,7){\colorbox{gray}}

\put(290,93){\colorbox{gray}}
\put(290,97){\colorbox{gray}}
\put(294,93){\colorbox{gray}}
\put(294,97){\colorbox{gray}}
\put(290,83){\colorbox{gray}}
\put(290,87){\colorbox{gray}}
\put(294,83){\colorbox{gray}}
\put(294,87){\colorbox{gray}}
\put(290,73){\colorbox{gray}}
\put(290,77){\colorbox{gray}}
\put(294,73){\colorbox{gray}}
\put(294,77){\colorbox{gray}}
\put(290,63){\colorbox{gray}}
\put(290,67){\colorbox{gray}}
\put(294,63){\colorbox{gray}}
\put(294,67){\colorbox{gray}}
\put(290,53){\colorbox{gray}}
\put(290,57){\colorbox{gray}}
\put(294,53){\colorbox{gray}}
\put(294,57){\colorbox{gray}}
\put(290,43){\colorbox{gray}}
\put(290,47){\colorbox{gray}}
\put(294,43){\colorbox{gray}}
\put(294,47){\colorbox{gray}}
\put(290,33){\colorbox{gray}}
\put(290,37){\colorbox{gray}}
\put(294,33){\colorbox{gray}}
\put(294,37){\colorbox{gray}}
\put(290,23){\colorbox{gray}}
\put(290,27){\colorbox{gray}}
\put(294,23){\colorbox{gray}}
\put(294,27){\colorbox{gray}}
\put(290,13){\colorbox{gray}}
\put(290,17){\colorbox{gray}}
\put(294,13){\colorbox{gray}}
\put(294,17){\colorbox{gray}}
\put(290,3){\colorbox{gray}}
\put(290,7){\colorbox{gray}}
\put(294,3){\colorbox{gray}}
\put(294,7){\colorbox{gray}}

\put(200,0){\framebox(10,10)}
\put(210,0){\framebox(10,10)}
\put(220,0){\framebox(10,10)}
\put(230,0){\framebox(10,10)}
\put(240,0){\framebox(10,10)}
\put(250,0){\framebox(10,10)}
\put(260,0){\framebox(10,10)}
\put(270,0){\framebox(10,10)}
\put(280,0){\framebox(10,10)}
\put(290,0){\framebox(10,10)}
\put(200,10){\framebox(10,10)}
\put(210,10){\framebox(10,10)}
\put(220,10){\framebox(10,10)}
\put(230,10){\framebox(10,10)}
\put(240,10){\framebox(10,10)}
\put(250,10){\framebox(10,10)}
\put(260,10){\framebox(10,10)}
\put(270,10){\framebox(10,10)}
\put(280,10){\framebox(10,10)}
\put(290,10){\framebox(10,10)}
\put(200,20){\framebox(10,10)}
\put(210,20){\framebox(10,10)}
\put(220,20){\framebox(10,10)}
\put(230,20){\framebox(10,10)}
\put(240,20){\framebox(10,10)}
\put(250,20){\framebox(10,10)}
\put(260,20){\framebox(10,10)}
\put(270,20){\framebox(10,10)}
\put(280,20){\framebox(10,10)}
\put(290,20){\framebox(10,10)}
\put(200,30){\framebox(10,10)}
\put(210,30){\framebox(10,10)}
\put(220,30){\framebox(10,10)}
\put(230,30){\framebox(10,10)}
\put(240,30){\framebox(10,10)}
\put(250,30){\framebox(10,10)}
\put(260,30){\framebox(10,10)}
\put(270,30){\framebox(10,10)}
\put(280,30){\framebox(10,10)}
\put(290,30){\framebox(10,10)}
\put(200,40){\framebox(10,10)}
\put(210,40){\framebox(10,10)}
\put(220,40){\framebox(10,10)}
\put(230,40){\framebox(10,10)}
\put(240,40){\framebox(10,10)}
\put(250,40){\framebox(10,10)}
\put(260,40){\framebox(10,10)}
\put(270,40){\framebox(10,10)}
\put(280,40){\framebox(10,10)}
\put(290,40){\framebox(10,10)}
\put(200,50){\framebox(10,10)}
\put(210,50){\framebox(10,10)}
\put(220,50){\framebox(10,10)}
\put(230,50){\framebox(10,10)}
\put(240,50){\framebox(10,10)}
\put(250,50){\framebox(10,10)}
\put(260,50){\framebox(10,10)}
\put(270,50){\framebox(10,10)}
\put(280,50){\framebox(10,10)}
\put(290,50){\framebox(10,10)}
\put(200,60){\framebox(10,10)}
\put(210,60){\framebox(10,10)}
\put(220,60){\framebox(10,10)}
\put(230,60){\framebox(10,10)}
\put(240,60){\framebox(10,10)}
\put(250,60){\framebox(10,10)}
\put(260,60){\framebox(10,10)}
\put(270,60){\framebox(10,10)}
\put(280,60){\framebox(10,10)}
\put(290,60){\framebox(10,10)}
\put(200,70){\framebox(10,10)}
\put(210,70){\framebox(10,10)}
\put(220,70){\framebox(10,10)}
\put(230,70){\framebox(10,10)}
\put(240,70){\framebox(10,10)}
\put(250,70){\framebox(10,10)}
\put(260,70){\framebox(10,10)}
\put(270,70){\framebox(10,10)}
\put(280,70){\framebox(10,10)}
\put(290,70){\framebox(10,10)}
\put(200,80){\framebox(10,10)}
\put(210,80){\framebox(10,10)}
\put(220,80){\framebox(10,10)}
\put(230,80){\framebox(10,10)}
\put(240,80){\framebox(10,10)}
\put(250,80){\framebox(10,10)}
\put(260,80){\framebox(10,10)}
\put(270,80){\framebox(10,10)}
\put(280,80){\framebox(10,10)}
\put(290,80){\framebox(10,10)}
\put(200,90){\framebox(10,10)}
\put(210,90){\framebox(10,10)}
\put(220,90){\framebox(10,10)}
\put(230,90){\framebox(10,10)}
\put(240,90){\framebox(10,10)}
\put(250,90){\framebox(10,10)}
\put(260,90){\framebox(10,10)}
\put(270,90){\framebox(10,10)}
\put(280,90){\framebox(10,10)}
\put(290,90){\framebox(10,10)}
\end{picture}
\end{center}
\caption{The pictures of $H(h)$ and $\p_\Theta$ for $h=(4,5,5,8,10)$ and $\Theta=\Delta \setminus \{\alpha_2,\alpha_4,\alpha_5,\alpha_8 \}$.}
\label{pic:h=(4,5,8,8,10)}
\end{figure}
\end{example}

\bigskip

\section{A summation formula for the Poincar\'e polynomial of $\Hess_\Theta(N,I)$} \label{sect:summation formula}

A summation formula for the Poincar\'e polynomials of regular nilpotent Hessenberg varieties $\Hess(N,I)$ is given by \cite{Pre13}. 
In this section we provide a summation formula for the Poincar\'e polynomials of regular nilpotent partial Hessenberg varieties $\Hess_\Theta(N,I)$.
 
\subsection{Affine paving}

We first review the result of an affine paving for the partial flag variety $G/P$. 
We refer the reader to \cite{BGG} for the discussion below. 

The Bruhat decomposition for $G$ gives the decomposition $G=\bigsqcup_{w \in W} BwB$. 
This yields the following decomposition of the flag variety $G/B$: 
\begin{align*} 
G/B = \bigsqcup_{w \in W} BwB/B. 
\end{align*}
Each $B$-orbit $BwB/B$ in $G/B$ is called the \emph{Schubert cell}.
We denote by $\ell(w)$ the length of $w$, namely the smallest integer $p$ of simple reflections needed to write $w=s_{i_1}s_{i_2} \dots s_{i_p}$. 
It is well-known that the Schubert cell $BwB/B$ is isomorphic to an affine space $\C^{\ell(w)}$. 
In fact, we write $U_w \coloneqq U \cap wU^{-}w^{-1}$ where $U^{-}=w_0Uw_0^{-1}$ and $w_0$ is the longest element in the Weyl group $W$. 
Then for $w \in W$, there is an isomorphism 
\begin{align} \label{eq:SchubertGB}
U_w \cong BwB/B,
\end{align}
which sends $u$ to $uwB$. 
It is also known that $U_w \cong \C^{\ell(w)}$ for $w \in W$ (e.g. \cite[Remarks in Section~14.12]{Bor91}). 

In more generality, the discussion above is generalized to the partial flag variety $G/P$ as follows.
Let $W^\Theta$ be the minimal left coset representatives, i.e. $W^\Theta \coloneqq \{w \in W \mid \ell(w) < \ell(ws_i) \ \textrm{for all} \ i \ \textrm{with} \ \alpha_i \in \Theta \}$. 
Then it is known that the partial flag variety $G/P$ has the following decomposition:
\begin{align*}
G/P = \bigsqcup_{w \in W^\Theta} BwP/P,
\end{align*}
where we have an isomorphism  
\begin{align} \label{eq:SchubertGP}
U_w \cong BwP/P \ \textrm{for} \ w \in W^\Theta,
\end{align}
which sends $u$ to $uwP$. 
In particular, $BwP/P$ is isomorphic to $\C^{\ell(w)}$ for $w \in W^\Theta$.
The decomposition above with the following fact gives a formula for the Poincar\'e polynomial for $G/P$. 

\begin{proposition} $($\cite[Section~B.3]{Ful97}, \cite[Examples~1.9.1 and 19.1.11]{Ful98}$)$ \label{proposition:paving}
If a complete algebraic variety $Z$ has a filtration $Z=Z_m \supset Z_{m-1} \supset \dots \supset Z_0 \supset Z_{-1}=\emptyset$ by closed subsets, with each $Z_i \setminus Z_{i-1}$ a disjoint union of $U_{i,j}$ isomorphic to an affine space $\C^{n(i,j)}$, then the homology classes $[\overline{U_{i,j}}]$ of the closures of $U_{i,j}$ forms a basis of the integral homology of $Z$. 
In particular, the integral cohomology of $Z$ is torsion-free and it is concentrated in even degree.  
\end{proposition}

In the setting of Proposition~\ref{proposition:paving}, we say that $Z$ is \emph{paved by affines} $\bigsqcup_{i=0}^m \big(\bigsqcup_{j} U_{i,j} \big)$.
The partial flag variety $G/P$ is paved by affines $\bigsqcup_{w \in W^\Theta} BwP/P$. 
In fact, we set $X_w^P \coloneqq \overline{BwP/P}$ for $w \in W^\Theta$ and $Z_i=\bigcup_{w \in W^\Theta \atop \ell(w)=i} X_w^P$ for $0 \leq i \leq \ell(w_0)$ where $\overline{BwP/P}$ denotes the closure of the Schubert cell $BwP/P$ which is called the \emph{Schubert variety} in $G/P$. As is well-known, the Schubert variety $X_w^P$ has the decomposition $X_w^P = \bigsqcup_{v \in W^\Theta \atop v \leq w} BvP/P$ where $v \leq w$ means the Bruhar order, namely a reduced decomposition for $v$ is a substring of some reduced decomposition for $w$.
Then we have $Z_i \setminus Z_{i-1}=\sqcup_{w \in W^\Theta \atop \ell(w)=i} BwP/P$, so $G/P$ is paved by affines $\bigsqcup_{w \in W^\Theta} BwP/P$. 
By Proposition~\ref{proposition:paving}, the Poincar\'e polynomial of $G/P$ is given by
\begin{align*}
\Poin(G/P, \sqrt{\q}) = \sum_{w \in W^\Theta} \q^{\ell(w)}. 
\end{align*}

Tymoczko proved that every Hessenberg variety in Lie type $A$ is paved by affines by \cite{Tym06}, and regular nilpotent Hessenberg varieties are paved by affines for classical Lie types in \cite{Tym07}. 
For arbitrary Lie types, this result was proved by Precup in \cite{Pre13}.
In order to state the result for the regular nilpotent case, we set 
\begin{align} \label{eq:Nw}
\NN(w) \coloneqq \{\alpha \in \Phi^+ \mid w(\alpha) \in \Phi^{-} \}
\end{align}
for $w \in W$ where $\Phi^{-}$ denotes the set of negative roots.

\begin{theorem} $($\cite[Theorem~4.10 and Corollary~4.13]{Pre13}$)$ \label{theorem:paving_Hessenberg}
Let $N_0$ be the regular nilpotent element defined in \eqref{eq:N0} and $I$ a lower ideal in $\Phi^+$. 
Let $\Hess(N_0,I)$ be the associated regular nilpotent Hessenberg variety in \eqref{eq:Hess(N,I)}. 
\begin{enumerate}
\item For all $w \in W$, the intersection $\Hess(N_0,I) \cap BwB/B$ is nonempty if and only if 
$w^{-1}(\Delta) \subset (-I) \cup \Phi^+$. 
\item If $\Hess(N_0,I) \cap BwB/B$ is nonempty, then 
\begin{align*}
\Hess(N_0,I) \cap BwB/B \cong \C^{|\NN(w) \cap I|}.
\end{align*}
Hence, the regular nilpotent Hessenberg variety $\Hess(N_0,I)$ is paved by affines $\bigsqcup_{w \in W \atop w^{-1}(\Delta) \subset (-I) \cup \Phi^+} \big( \Hess(N_0,I) \cap BwB/B \big)$. 
\item Let $N$ be a regular nilpotent element in $\g$. 
Then the Poincar\'e polynomial of $\Hess(N,I)$ is equal to 
\begin{align} \label{eq:PoincarePolynomialHessGB}
\Poin(\Hess(N,I), \sqrt{\q}) = \sum_{w \in W \atop w^{-1}(\Delta) \subset (-I) \cup \Phi^+} \q^{|\NN(w) \cap I|}. 
\end{align}
\end{enumerate}
\end{theorem}

\begin{remark}
It is written as $\Hess(N_0,I) \cap BwB/B \cong \C^{|\NN(w^{-1}) \cap w(-I)|}$ in \cite{Pre13}, but we describe it as $\Hess(N_0,I) \cap BwB/B \cong \C^{|\NN(w) \cap I|}$ above since we have $|\NN(w^{-1}) \cap w(-I)|=|\NN(w) \cap I|$.
\end{remark}

The proposition below is immediately follows from Theorem~\ref{theorem:paving_Hessenberg}.

\begin{proposition} 
Let $N_0$ be the regular nilpotent element defined in \eqref{eq:N0} and $I$ a $\Theta$-ideal in $\Phi^+$.
Let $\Hess_\Theta(N_0,I)$ denote the associated regular nilpotent partial Hessenberg variety in \eqref{eq:Hess(N,I)Theta}. 
Then the following holds.
\begin{enumerate} 
\item For any $w \in W^\Theta$, $\Hess_\Theta(N_0,I) \cap BwP/P$ is nonempty if and only if 
$w^{-1}(\Delta) \subset (-I) \cup \Phi^+$. 
\item If $\Hess_\Theta(N_0,I) \cap BwP/P$ is nonempty, then 
\begin{align*}
\Hess_\Theta(N_0,I) \cap BwP/P \cong \C^{|\NN(w) \cap I|}.
\end{align*}
Hence, the regular nilpotent partial Hessenberg variety $\Hess_\Theta(N_0,I)$ is paved by affines $\bigsqcup_{w \in W^\Theta \atop w^{-1}(\Delta) \subset (-I) \cup \Phi^+} \big( \Hess_\Theta(N_0,I) \cap BwP/P \big)$. 
\item Let $N \in \g$ be a regular nilpotent element.
The Poincar\'e polynomial for $\Hess_\Theta(N,I)$ is  
\begin{align} \label{eq:PoincarePolynomialSum1}
\Poin(\Hess_\Theta(N,I), \sqrt{\q}) = \sum_{w \in W^\Theta \atop w^{-1}(\Delta) \subset (-I) \cup \Phi^+} \q^{|\NN(w) \cap I|}. 
\end{align}
\end{enumerate}
\end{proposition}

\begin{proof}
(1) For $w \in W^\Theta$, we have 
\begin{align*}
\Hess_\Theta(N_0,I) \cap BwP/P \neq \emptyset & \iff \textrm{there exists} \ b \in B \ \textrm{such that} \ \Ad((bw)^{-1})(N_0) \in H_I \\
& \iff \Hess(N_0,I) \cap BwB/B \neq \emptyset \\
& \iff w^{-1}(\Delta) \subset (-I) \cup \Phi^+ \ \ \ \textrm{(by Theorem~\ref{theorem:paving_Hessenberg}-(1))}.
\end{align*}

(2) If $\Hess_\Theta(N_0,I) \cap BwP/P$ is nonempty, then the intersection $\Hess_\Theta(N_0,I) \cap BwP/P$ is isomorphic to
\begin{align*}
\{ u \in U_w \mid \Ad\big((uw)^{-1}\big)(N_0) \in H_I \}
\end{align*}
under the identification \eqref{eq:SchubertGP}. 
This is also identified with $\Hess(N_0,I) \cap BwB/B$ under \eqref{eq:SchubertGB}. 
Thus, the desired isomorphism $\Hess_\Theta(N_0,I) \cap BwP/P \cong \C^{|\NN(w) \cap I|}$ follows from Theorem~\ref{theorem:paving_Hessenberg}-(2). 
By setting $Z_i=\bigcup_{w \in W^\Theta \atop \ell(w)=i} \big(X_w^P \cap \Hess_\Theta(N_0,I) \big)$ in Proposition~\ref{proposition:paving}, one can see from (1) above that the regular nilpotent partial Hessenberg variety $\Hess_\Theta(N_0,I)$ is paved by affines $\bigsqcup_{w \in W^\Theta \atop w^{-1}(\Delta) \subset (-I) \cup \Phi^+} \big(BwP/P \cap \Hess_\Theta(N_0,I) \big)$. 

(3) Since $\Hess_\Theta(N,I) \cong \Hess_\Theta(N_0,I)$ by \eqref{eq:Hess(N0,H)Theta}, we obtain \eqref{eq:PoincarePolynomialSum1} from Proposition~\ref{proposition:paving} with (2) above.
\end{proof}

\begin{remark}
Fresse gives a result for an affine paving in more general setting in \cite{Fre}. 
\end{remark}

\begin{example} \label{example:PoincarePolynomial_h=(4,5,5,8,10)}
We consider the case of type $A_{n-1}$. 
We use the notations explained in Section~\ref{subsection:Type A description}. 
The Weyl group $W$ is the permutation group $\mathfrak{S}_n$ on $[n]$ and $W^\Theta$ is the set of permutations $w \in \mathfrak{S}_n$ such that $w(\kk_{p-1}+1) < w(\kk_{p-1}+2) < \cdots < w(\kk_p)$ for all $1 \leq p \leq s+1$. 
Here we recall that $\Theta=\{ \alpha_i \in \Delta \mid i \in [n-1]\setminus \KK \}$ and $\KK = \{\kk_1,\ldots,\kk_s \}$ with the convention that $\kk_0=0$ and $\kk_{s+1}=n$. 
One can also see that $w^{-1}(\Delta) \subset (-I) \cup \Phi^+$ if and only if $w^{-1}(w(j)-1) \leq \tilde{h}(j)$ for all $j \in [n]$ where $\tilde{h}$ is defined in \eqref{eq:tilde h} and we take the convention that $w(0)=0$ (cf. \cite[Proposition~2.9]{AHHM} and \cite[Proposition~5.2]{HaTy17}). 

For instance, we take $\KK=\{2,4,5,8\}$ in type $A_9$ and $h=(4,5,5,8,10)$ given in Example~\ref{example:h=(4,5,5,8,10)}. 
Note that the extended Hessenberg function is $\tilde{h}=(4,4,5,5,5,8,8,8,10,10)$. 
Then the permutations $w \in W^\Theta$ such that $w^{-1}(\Delta) \subset (-I) \cup \Phi^+$ consists of the following 18 permutations in one-line notation.
\begin{align*}
&1\,2\,3\,4\,5\,6\,7\,8\,9\,10, \ 1\,2\,3\,5\,4\,6\,7\,8\,9\,10, \ 1\,2\,4\,5\,3\,6\,7\,8\,9\,10, \ 1\,3\,2\,4\,5\,6\,7\,8\,9\,10, \ 1\,3\,2\,5\,4\,6\,7\,8\,9\,10, \\
&1\,4\,2\,3\,5\,6\,7\,8\,9\,10, \ 1\,4\,3\,5\,2\,6\,7\,8\,9\,10, \ 1\,5\,2\,4\,3\,6\,7\,8\,9\,10, \ 1\,5\,3\,4\,2\,6\,7\,8\,9\,10, \ 2\,3\,1\,4\,5\,6\,7\,8\,9\,10, \\
&2\,3\,1\,5\,4\,6\,7\,8\,9\,10, \ 2\,4\,1\,3\,5\,6\,7\,8\,9\,10, \ 2\,5\,1\,4\,3\,6\,7\,8\,9\,10, \ 3\,4\,1\,2\,5\,6\,7\,8\,9\,10, \ 3\,4\,2\,5\,1\,6\,7\,8\,9\,10, \\ 
&3\,5\,2\,4\,1\,6\,7\,8\,9\,10, \ 4\,5\,1\,3\,2\,6\,7\,8\,9\,10, \ 4\,5\,2\,3\,1\,6\,7\,8\,9\,10.  
\end{align*}
In this case, the formula \eqref{eq:PoincarePolynomialSum1} for the Poincar\'e polynomial gives 
\begin{align} \label{eq:PoincarePolynomial_h=(4,5,5,8,10)}
\Poin(\Hess_\KK(N,h), \sqrt{\q}) = 1 + 2\q +4\q^2 + 4\q^3 + 4 \q^4 +2\q^5 + \q^6.  
\end{align}
\end{example}

\subsection{Weyl type subset} \label{subsection:Weyl type subset}

Let $I$ be a lower ideal in $\Phi^+$. 
A subset $Y \subset I$ is of \emph{Weyl type} if $\alpha, \beta \in Y$ and $\alpha+\beta \in I$, then $\alpha+\beta \in Y$, and if $\gamma, \delta \in I \setminus Y$ and $\gamma + \delta \in I$, then $\gamma + \delta \in I \setminus Y$.
This notion was introduced by Sommers and Tymoczko \cite{SoTy}. 
The set of the Weyl type subsets of $I$ is denoted by $\mathcal{W}^I$.
 
For all $w \in W$, the set $\NN(w) \cap I$ is a Weyl type subset of $I$ where $\NN(w)$ is defined in \eqref{eq:Nw}.

\begin{proposition} $($\cite[Proposition~6.3]{SoTy}$)$ \label{proposition:one-to-oneWeylType}
Let $I$ be a lower ideal in $\Phi^+$. 
There is a bijection 
\begin{align*}
\eta: \{w \in W \mid w^{-1}(\Delta) \subset (-I) \cup \Phi^+ \} \to \mathcal{W}^I
\end{align*}
which sends $w$ to $\NN(w) \cap I$.
In paticular, one can write \eqref{eq:PoincarePolynomialHessGB} as 
\begin{align} \label{eq:PoincarePolynomialHessGBWeylType}
\Poin(\Hess(N,I), \sqrt{\q}) = \sum_{Y \in \mathcal{W}^I} \q^{|Y|}.
\end{align} 
\end{proposition}

We generalize the formula in \eqref{eq:PoincarePolynomialHessGBWeylType} to the formula for regular nilpotent partial Hessenberg varieties. 
For this purpose, we set
\begin{align} \label{eq:WeylTypeTheta}
\mathcal{W}^{I, \Theta} \coloneqq \{Y \in \mathcal{W}^I \mid Y \cap \Phi^+_\Theta = \emptyset \}
\end{align}
for a $\Theta$-ideal $I$.

\begin{lemma} $($\cite[Lemma~5.1]{SoTy}$)$ \label{lemma:Nw}
Let $u,v \in W$.
Then, $\NN(v) \subset \NN(uv)$ if and only if $\ell(uv)=\ell(u)+\ell(v)$. 
\end{lemma}

\begin{lemma} \label{lemma:WeylTypeNw} 
Let $\Theta \subset \Delta$ and $w \in W$.
Then, $w \in W^\Theta$ if and only if $\NN(w) \cap \Phi^+_\Theta = \emptyset$.
\end{lemma}

\begin{proof}
Let $w \in W^\Theta$. 
Suppose in order to obtain a contradiction that $\NN(w) \cap \Phi^+_\Theta \neq \emptyset$.
We take a root $\alpha \in \NN(w) \cap \Phi^+_\Theta$. 
Since $\alpha \in \Phi^+_\Theta$, we have $s_\alpha \in W_\Theta$ (e.g. \cite[Proposition~1.10-(a)]{Hum90}). 
Hence, we have $\ell(ws_\alpha) = \ell(w) + \ell(s_\alpha)$ (\cite[Proposition~1.10-(c)]{Hum90}). 
It follows from Lemma~\ref{lemma:Nw} that $\NN(s_\alpha) \subset \NN(ws_\alpha)$.
This inclusion yields a contradiction. 
In fact, it is clear that $\alpha \in \NN(s_\alpha)$. 
On the other hand, $ws_\alpha(\alpha)=-w(\alpha)$ is a positive root since $\alpha \in \NN(w)$. 
This means that $\alpha \notin \NN(ws_\alpha)$, which gives a contradiction.  
Therefore, we obtain $\NN(w) \cap \Phi^+_\Theta = \emptyset$.

Conversely, we assume that $\NN(w) \cap \Phi^+_\Theta = \emptyset$. 
By \cite[Proposition~1.10-(c)]{Hum90}, there is a unique $u \in W^\Theta$ and a unique $v \in W_\Theta$ such that $w=uv$. 
We also have $\ell(w)=\ell(u)+\ell(v)$. 
Thus, we obtain $\NN(v) \subset \NN(w)$ from Lemma~\ref{lemma:Nw}.
This with the assumption $\NN(w) \cap \Phi^+_\Theta = \emptyset$ yields $\NN(v) \cap \Phi^+_\Theta = \emptyset$.
In other words, for any $\alpha \in \Phi^+_\Theta$, $v(\alpha)$ is a positive root.
This means that $v(\Phi^+_\Theta) = \Phi^+_\Theta$ since $v \in W_\Theta$ (cf. \cite[Proposition~1.10-(a)]{Hum90}).
By \cite[Theorem~1.8]{Hum90}, $v$ is the identity element of $W_\Theta$ and hence we have $w=uv=u \in W^\Theta$ as desired.
\end{proof}

\begin{proposition} \label{proposition:one-to-oneWeylTypeTheta}
Let $\Theta \subset \Delta$ and $I$ a $\Theta$-ideal in $\Phi^+$. 
Then there is a bijection 
\begin{align*}
\eta_\Theta: \{w \in W^\Theta \mid w^{-1}(\Delta) \subset (-I) \cup \Phi^+ \} \to \mathcal{W}^{I, \Theta}
\end{align*}
which sends $w$ to $\NN(w) \cap I$.
\end{proposition}

\begin{proof}
Consider the bijection $\eta$ in Proposition~\ref{proposition:one-to-oneWeylType}. 
Since $I \supset \Phi^+_\Theta$ by a condition of a $\Theta$-ideal $I$, we have $\eta(w) \cap \Phi^+_\Theta=\NN(w) \cap I \cap \Phi^+_\Theta = \NN(w) \cap \Phi^+_\Theta$.
Thus, it follows from Lemma~\ref{lemma:WeylTypeNw} that $\eta(w) \in \mathcal{W}^{I, \Theta}$ if and ony if $w \in W^\Theta$.
Therefore, the bijection $\eta$ induces the bijection $\eta_\Theta$.
\end{proof}

\begin{theorem} \label{theorem:PoincarePolynomialHessGPWeylType}
Let $N$ be a regular nilpotent element in $\g$ and $I$ a $\Theta$-ideal.
Let $\Hess_\Theta(N,I)$ be the associated regular nilpotent partial Hessenberg variety in \eqref{eq:Hess(N,I)Theta}. 
Then we have 
\begin{align*} 
\Poin(\Hess_\Theta(N,I), \sqrt{\q}) = \sum_{Y \in \mathcal{W}^{I, \Theta}} \q^{|Y|}
\end{align*} 
where $\mathcal{W}^{I, \Theta}$ is defined in \eqref{eq:WeylTypeTheta}.
\end{theorem}

\begin{proof}
The result follows from \eqref{eq:PoincarePolynomialSum1} and Proposition~\ref{proposition:one-to-oneWeylTypeTheta}. 
\end{proof}

\bigskip

\section{A product formula for the Poincar\'e polynomial of $\Hess_\Theta(N,I)$} \label{sect:product formula}

We see the summation formula in \eqref{eq:PoincarePolynomialHessGB}, while the Poincar\'e polynomial of a regular nilpotent Hessenberg variety $\Hess(N,I)$ is also discribed as
\begin{align} \label{eq:PoincarePolynomialHessGBproduct}
\Poin(\Hess(N,I),\sqrt{\q}) = \prod_{\alpha \in I} \frac{1-\q^{\height(\alpha)+1}}{1-\q^{\height(\alpha)}}. 
\end{align}
The product formula above was announced by Dale Peterson in \cite[Theorem~3]{BrCa} without proof and it was proved by \cite{AHMMS} from the point of view of hyperplane arrangements. 
In this section we generalize the formula for $\Hess(N,I)$ to the formula of regular nilpotent partial Hessenberg varieties $\Hess_\Theta(N,I)$.

\subsection{Fiber bundle}
We first recall that the natural projection $\pi: G/B \to G/P$ is a fiber bundle with fiber $P/B$.
In fact, since the natural map $G \to G/P$ is a principal $P$-bundle, $G \times_P P/B$ is a fiber bundle over $G/P$ with fiber $P/B$.
Here, $G \times_P P/B$ is defined by the quotient of the direct product $G \times P/B$ by the left $P$-action given by $p_1 \cdot (g,p_2B) = (gp_1^{-1},p_1p_2B)$ for $g \in G$, $p_2B \in P/B$, and $p_1 \in P$. 
Then we have the identification $G/B \cong G \times_P P/B$ which sends $gB$ to $[g,eB]$ where $e$ denotes the identity element of $P$. 
Hence, we conclude that the natural projection $\pi: G/B \to G/P$ is a fiber bundle with fiber $P/B$. 

The following lemma is a key property to study regular nilpotent partial Hessenberg varieties. 
We remark that it was discussed in \cite[Definition~2.4 and surrounding discussion]{KiLe} for regular semisimple partial Hessenberg varieties in type $A$.

\begin{lemma} \label{lemma:fiber bundle}
Let $N$ be a regular nilpotent element of $\g$ and $I$ a $\Theta$-ideal in $\Phi^+$.   
Then, the natural projection $\pi_I: \Hess(N,I) \to \Hess_\Theta(N,I)$ is a fiber bundle with fiber $P/B$. 
\end{lemma}

\begin{proof}
Consider the following commutative diagram
\begin{center}
\begin{tikzcd}
G/B \arrow[r, "\pi"] &[0.5em] G/P \\
\Hess(N,I) \arrow[r, "\pi_I"']  \arrow[hookrightarrow,u, "{\rm inclusion}"]                                         & \Hess_\Theta(N,I). \arrow[hookrightarrow,u, "{\rm inclusion}"'] 
\end{tikzcd}
\end{center}
It suffices to show that $\pi_I^{-1}(gP)$ is isomorphic to $P/B$ for any $gP \in \Hess_\Theta(N,I)$ since the local triviality is inherited from that of the fiber bundle $\pi: G/B \to G/P$ with fiber $P/B$.
By setting 
\begin{align} \label{eq:GNHI}
G_{N,H_I} \coloneqq \{g \in G \mid \Ad(g^{-1})(N) \in H_I \}, 
\end{align}
the regular nilpotent partial Hessenberg variety $\Hess_\Theta(N,I)$ is identified with the quotient space $G_{N,H_I}/P$. 
Then the fiber of $\pi_I: G_{N,H_I}/B \to G_{N,H_I}/P$ at $g_1P \in G_{N,H_I}/P$ is  
\begin{align*}
\pi_I^{-1}(g_1P) = \{gB \in G_{N,H_I}/B \mid g_1^{-1}g \in P \} = g_1(P/B) \cong P/B.
\end{align*}
This completes the proof. 
\end{proof}

Recall that we denote by $X_w^P$ the Schubert variety associated with $w \in W^\Theta$ in $G/P$. 
For simplicity, we write $X_w$ for the Schubert variety associated with $w \in W$ in $G/B$. 
The first equality in the following lemma is a well-known fact. 

\begin{lemma} \label{lemma:P/B}
Let $N_0$ be the regular nilpotent element of $\g$ defined in \eqref{eq:N0}. 
Then we have the follwoing equalities: 
\begin{align*}
P/B = X_{w_0^{\Theta}} = \Hess(N_0,\Phi^+_\Theta) (= \Hess(N_0,\p))
\end{align*}
where $X_{w_0^{\Theta}}$ is the Schubert variety in $G/B$ associated with the longest element  $w_0^{\Theta}$ of $W_\Theta$. 
\end{lemma}

\begin{proof}
It is known that $P=P_\Theta=BW_{\Theta}B$ (e.g. \cite[Proposition~14.18]{Bor91} or \cite[Section~30]{Hum75}), so it holds that $P/B = \sqcup_{w \in W_\Theta} BwB/B = X_{w_0^{\Theta}}$.
Let $G_{N_0,H_I}$ be the set in \eqref{eq:GNHI} for $N=N_0$. 
Then we have that 
\begin{align} \label{eq:PsubsetGHI}
P \subset G_{N_0,H_I}
\end{align}
for any $\Theta$-ideal $I \subset \Phi^+$. 
In fact, if $x \in H_I$ and $g \in P$, then $\Ad(g^{-1})(x) \in H_I$ since $H_I$ is a $\p$-submodule.
The inclusion \eqref{eq:PsubsetGHI} yields that $P/B \subset G_{N_0,H_I}/B = \Hess(N_0,I)$ for arbitarary $\Theta$-ideal $I \subset \Phi^+$.
In particular, $P/B \subset \Hess(N_0,\Phi^+_\Theta) (= \Hess(N_0,\p))$.  
By Proposition~\ref{proposition:propertiesHess(N,I)} one has $\dim_\C P/B = \dim_\C \p/\b = |\Phi^+_\Theta| = \dim_\C \Hess(N_0,\Phi^+_\Theta)$ and $\Hess(N_0,\Phi^+_\Theta)$ is irreducible, so we conclude that $P/B = \Hess(N_0,\Phi^+_\Theta)$ as desired.
\end{proof}

\begin{remark}
It is also well-known that $P/B$ is isomorphic to $\Levi_\Theta/B_\Theta$ where $B_\Theta \coloneqq \Levi_\Theta \cap B$. 
In fact, one can easily see that $P/B$ is the image of the closed embedding $\Levi_\Theta/B_\Theta \hookrightarrow G/B$ by a similar argument of Lemma~\ref{lemma:P/B}. 
\end{remark}

\begin{proposition} \label{proposition:propertyHess(N,I)Theta}
Let $N$ be a regular nilpotent element in $\g$ and $I$ a $\Theta$-ideal in $\Phi^+$.
Let $\Hess_\Theta(N,I)$ denotes the associated regular nilpotent partial Hessenberg variety in \eqref{eq:Hess(N,I)Theta}. 
Then the following holds.
\begin{enumerate}
\item $\Hess_\Theta(N,I)$ is irreducible. 
\item The complex dimension is given by $\dim_\C \Hess_\Theta(N,I) = |I| - |\Phi^+_\Theta|$. 
\item The Poincar\'e polynomial is discribed as
\begin{align*}
\Poin(\Hess_\Theta(N,I),\sqrt{\q}) = \prod_{\alpha \in I \setminus \Phi^+_\Theta} \frac{1-\q^{\height(\alpha)+1}}{1-\q^{\height(\alpha)}}. 
\end{align*}
\end{enumerate}
\end{proposition}

\begin{proof}
(1)  $\Hess(N,I)$ is irreducible by Proposition~\ref{proposition:propertiesHess(N,I)} and $\pi_I: \Hess(N,I) \to \Hess_\Theta(N,I)$ is surjective, so $\Hess_\Theta(N,I)$ is also irreducible.

(2) It follows from Lemma~\ref{lemma:fiber bundle} that $\dim_\C \Hess_\Theta(N,I) = \dim_\C \Hess(N,I) - \dim_\C P/B = |I| - |\Phi^+_\Theta|$. 
Here, we used Proposition~\ref{proposition:propertiesHess(N,I)} for the equality $\dim_\C \Hess(N,I)=|I|$.

(3) Consider the following commutative diagram
\begin{center}
\begin{tikzcd}
P/B \arrow[r, "\iota"] &[0.5em] G/B \arrow[r, "\pi"] &[0.5em] G/P \\
P/B  \arrow[r, "\iota_I"] \arrow[u,equal] &[0.5em]  \Hess(N,I) \arrow[r, "\pi_I"]  \arrow[hookrightarrow,u, "j_I"']  & \Hess_\Theta(N,I) \arrow[hookrightarrow,u, "j_{I, \Theta}"'] 
\end{tikzcd}
\end{center}
where the horizontal arrows denote the fiber bundles by Lemma~\ref{lemma:fiber bundle}. 
This induces the following commutative diagram:
\begin{center}
\begin{tikzcd}
H^*(G/P) \arrow[r, "\pi^*"] \arrow[rightarrow,d, "j_{I, \Theta}^*"'] &[0.5em] H^*(G/B) \arrow[rightarrow,d, "j_I^*"'] \arrow[r, "\iota^*"] &[0.5em] H^*(P/B) \\
H^*(\Hess_\Theta(N,I))  \arrow[r, "\pi_I^*"] &[0.5em]  H^*(\Hess(N,I)) \arrow[r, "\iota_I^*"]  & H^*(P/B). \arrow[u,equal]  
\end{tikzcd}
\end{center}
As is well-known, the restriction map from the cohomology of $G/B$ to the cohomology of a Schubert variety is surjective (cf.\cite{BGG}). 
This fact with Lemma~\ref{lemma:P/B} yields that $\iota^*$ is surjective.
By the commutative diagram above, $\iota_I^*$ is also surjective.
Hence, by the Leray--Hirsch theorem we deduce that 
\begin{align*}
\Poin(\Hess(N,I),\sqrt{\q}) &= \Poin(P/B,\sqrt{\q}) \Poin(\Hess_\Theta(N,I),\sqrt{\q}) \\
&= \Poin(\Hess(N_0,\Phi^+_\Theta),\sqrt{\q}) \Poin(\Hess_\Theta(N,I),\sqrt{\q})
\end{align*}
where we used Lemma~\ref{lemma:P/B} for the last equality above.
Therefore, we conclude from \eqref{eq:PoincarePolynomialHessGBproduct} that 
\begin{align*}
\Poin(\Hess_\Theta(N,I),\sqrt{\q}) = \prod_{\alpha \in I \setminus \Phi^+_\Theta} \frac{1-\q^{\height(\alpha)+1}}{1-\q^{\height(\alpha)}}. 
\end{align*}
\end{proof}

\begin{remark}
In type $A_{n-1}$, set $\KK = \{\kk_1,\ldots,\kk_s \}$ with $0=\kk_0 < \kk_1 < \kk_2 < \cdots < \kk_s < \kk_{s+1}=n$. 
Let $h: \KK \cup \{n \} \to \KK \cup \{n \}$ be a partial Hessenberg function. 
Then, one can see that Proposition~\ref{proposition:propertyHess(N,I)Theta}-(2) for type $A_{n-1}$ is described as 
\begin{align*}
\dim_\C \Hess_\KK(N,h) = \sum_{p=1}^{s} (\kk_p-\kk_{p-1}) (h(\kk_p)-\kk_p).
\end{align*}
Remark that the formula for regular semisimple partial Hessenberg varieties in type $A_{n-1}$ was given in \cite[Theorem~2.3]{KiLe}.
\end{remark}

\subsection{Height distribution} \label{subsection:Height distribution}

For a subset $Y \subset \Phi^+$, we define
\begin{align*}
\lambda_i^Y \coloneqq |\{\alpha \in Y \mid \height(\alpha)=i \}| \ \ \ \textrm{for} \ 1 \leq i \leq \rr
\end{align*}
where $\rr \coloneqq \max\{\height(\alpha) \mid \alpha \in Y \}$.
We call a sequence $(\lambda_1^Y, \lambda_2^Y, \cdots, \lambda_\rr^Y)$ the \emph{height distribution in $Y$}. 
We set an integer 
\begin{align} \label{eq:kiY}
\mm_i^Y \coloneqq \lambda_i^Y-\lambda_{i+1}^Y \ \ \ \textrm{for} \ 1 \leq i \leq \rr 
\end{align}
with the convention $\lambda_{\rr+1}^Y \coloneqq 0$.
It is known that $\lambda_1^I \geq \lambda_2^I \geq \cdots \geq \lambda_\rr^I$ for arbitrary lower ideal $I \subset \Phi^+$ by \cite[Proposition~3.1]{SoTy}, which implies that $\mm_i^I \geq 0$ for all $1 \leq i \leq \rr$. 
Then one can easily see that the formula in \eqref{eq:PoincarePolynomialHessGBproduct} can be written as follows (cf. \cite{AHMMS}):
\begin{align} \label{eq:PoincarePolynomialHessGBHeightDistribution}
\Poin(\Hess(N,I),\sqrt{\q}) = \prod_{i=1}^\rr (1+\q+\q^2+\cdots+\q^i)^{\mm_i^I}.
\end{align}

\begin{example}
We consider a lower ideal $I$ below in type $A_9$:
\begin{align*}
I = \{&x_1-x_2, x_1-x_3, x_1-x_4, x_2-x_3, x_2-x_4, x_3-x_4, x_3-x_5, x_4-x_5, \\
&x_6-x_7, x_6-x_8, x_7-x_8, x_9-x_{10} \}.
\end{align*}
The height distribution in $I$ is $(\lambda_1^I,\lambda_2^I,\lambda_3^I)=(7,4,1)$ as shown in Figure~\ref{pic:Young diagram(7,4,1)}.
\begin{figure}[h]
\begin{center}
\begin{picture}(230,45)
\put(0,0){\framebox(40,15){$x_1-x_2$}}
\put(40,0){\framebox(40,15){$x_2-x_3$}}
\put(80,0){\framebox(40,15){$x_3-x_4$}}
\put(120,0){\framebox(40,15){$x_4-x_5$}}
\put(160,0){\framebox(40,15){$x_6-x_7$}}
\put(200,0){\framebox(40,15){$x_7-x_8$}}
\put(240,0){\framebox(42,15){$x_9-x_{10}$}}
\put(0,15){\framebox(40,15){$x_1-x_3$}}
\put(40,15){\framebox(40,15){$x_2-x_4$}}
\put(80,15){\framebox(40,15){$x_3-x_5$}}
\put(120,15){\framebox(40,15){$x_6-x_8$}}
\put(0,30){\framebox(40,15){$x_1-x_4$}}

\put(-50,4){{\footnotesize $\lambda_1^I=7$}}
\put(-50,19){{\footnotesize $\lambda_2^I=4$}}
\put(-50,34){{\footnotesize $\lambda_3^I=1$}}
\end{picture}
\end{center}
\caption{The height distribution in $I$.}
\label{pic:Young diagram(7,4,1)}
\end{figure}

Since $(\mm_1^I,\mm_2^I,\mm_3^I)=(3,3,1)$, the Poincar\'e polynomial of $\Hess(N,I)$ is 
\begin{align*}
\Poin(\Hess(N,I),\sqrt{\q}) = (1+\q)^3(1+\q+\q^2)^3(1+\q+\q^2+\q^3).
\end{align*}
\end{example}

\begin{theorem} \label{theorem:height_distributionHessTheta}
Let $N$ be a regular nilpotent element in $\g$. 
For a $\Theta$-ideal $I \subset \Phi^+$, a sequence $\big(\lambda_1^{I \setminus \Phi^+_\Theta}, \lambda_2^{I \setminus \Phi^+_\Theta}, \cdots, \lambda_\rr^{I \setminus \Phi^+_\Theta} \big)$ denotes the height distribution in ${I \setminus \Phi^+_\Theta}$ where $\rr=\max\{\height(\alpha) \mid \alpha \in I \setminus \Phi^+_\Theta \}$. 
Let $\mm_i^{I \setminus \Phi^+_\Theta} \ (1 \leq i \leq \rr)$ be an integer defined in \eqref{eq:kiY}. 
Then the Poincar\'e polynomial of the regular nilpotent partial Hessenberg variety $\Hess_\Theta(N,I)$ is equal to 
\begin{align} \label{eq:height_distributionHessTheta}
\Poin(\Hess_\Theta(N,I), \sqrt{\q}) = \prod_{i=1}^\rr (1+\q+\q^2+\cdots+\q^i)^{\mm_i^{I \setminus \Phi^+_\Theta}}.
\end{align} 
\end{theorem}

\begin{proof}
As seen in the proof of Proposition~\ref{proposition:propertyHess(N,I)Theta}-(3), we obtain
\begin{align*}
\Poin(\Hess(N,I),\sqrt{\q}) = \Poin(\Hess(N_0,\Phi^+_\Theta),\sqrt{\q}) \Poin(\Hess_\Theta(N,I),\sqrt{\q}).
\end{align*}
This with \eqref{eq:PoincarePolynomialHessGBHeightDistribution} yields that 
\begin{align} \label{eq:proof_height_distribution}
\Poin(\Hess_\Theta(N,I),\sqrt{\q}) = \frac{\prod_{i=1}^{\rr'} (1+\q+\q^2+\cdots+\q^i)^{\mm_i^I}}{\prod_{i=1}^{\rr''} (1+\q+\q^2+\cdots+\q^i)^{\mm_i^{\Phi^+_\Theta}}}
\end{align}
where $\rr' \coloneqq \max\{\height(\alpha) \mid \alpha \in I \}$ and $\rr'' \coloneqq \max\{\height(\alpha) \mid \alpha \in \Phi^+_\Theta \}$.
Note that $\rr=\max\{\height(\alpha) \mid \alpha \in I \setminus \Phi^+_\Theta \} \leq \rr'$ and $\rr'' \leq \rr'$ by the condition $I \supset \Phi^+_\Theta$. 
We set $\lambda_i^{I \setminus \Phi^+_\Theta} =0$ for $i > \rr$ and $\lambda_i^{\Phi^+_\Theta} = 0$ for $i > \rr''$. 
Then, by the definition of the height distribution, we have $\lambda_i^{I \setminus \Phi^+_\Theta} = \lambda_i^{I} - \lambda_i^{\Phi^+_\Theta}$ for $1 \leq i \leq \rr'$. 
This implies that $\mm_i^{I \setminus \Phi^+_\Theta} = \mm_i^{I} - \mm_i^{\Phi^+_\Theta}$ for $1 \leq i \leq \rr'$. 
Therefore, the right hand side of \eqref{eq:proof_height_distribution} is 
\begin{align*}
\prod_{i=1}^{\rr'} (1+\q+\q^2+\cdots+\q^i)^{\mm_i^I-\mm_i^{\Phi^+_\Theta}} = \prod_{i=1}^{\rr} (1+\q+\q^2+\cdots+\q^i)^{\mm_i^{I \setminus \Phi^+_\Theta}}
\end{align*}
where we used the equality $\mm_i^{I \setminus \Phi^+_\Theta} = \lambda_i^{I \setminus \Phi^+_\Theta} - \lambda_{i+1}^{I \setminus \Phi^+_\Theta} =0$ whenever $i >\rr$.
\end{proof}

The right hand side of \eqref{eq:height_distributionHessTheta} seems to be a rational function since $\mm_i^{I \setminus \Phi^+_\Theta}$ may be taken as a negative integer, but it is a polynomial in the variable $\q$ because the left hand side of \eqref{eq:height_distributionHessTheta} is a polynomial. 
In other words, the numerator is divisible by the denominator in the right hand side of \eqref{eq:height_distributionHessTheta}. 
We give an example below. 
In the next subsection, Section~\ref{subsection:typeAformula}, we will explicitly give a polynomial description for \eqref{eq:height_distributionHessTheta} in type $A$.

\begin{example} \label{example:h=(4,5,5,8,10)PoincarePolynomial}
We consider the Poincar\'e polynomial of $\Hess_\KK(N,h)$ of type $A_9$ for $\KK=\{2,4,5,8\}$  and $h=(4,5,5,8,10)$, which was also considered in Example~\ref{example:PoincarePolynomial_h=(4,5,5,8,10)}. 
The corresponding $\Theta$-ideal $I$ and $\Phi^+_\Theta$ are described as Example~\ref{example:h=(4,5,5,8,10)}. 
The height distribution in $I \setminus \Phi^+_\Theta$ is $(\lambda_1^{I \setminus \Phi^+_\Theta},\lambda_2^{I \setminus \Phi^+_\Theta},\lambda_3^{I \setminus \Phi^+_\Theta})=(2,3,1)$ and tuple of integers is $(\mm_1^{I \setminus \Phi^+_\Theta},\mm_2^{I \setminus \Phi^+_\Theta},\mm_3^{I \setminus \Phi^+_\Theta})=(-1,2,1)$ as shown in Figure~\ref{pic:skew Young diagram(2,3,1)}, so the Poincar\'e polynomial of $\Hess_\KK(N,h)$ is given by
\begin{align*}
\Poin(\Hess_\KK(N,h),\sqrt{\q}) = \frac{1}{1+\q}(1+\q+\q^2)^2(1+\q+\q^2+\q^3) = (1+\q+\q^2)^2(1+\q^2), 
\end{align*}
which coincides with \eqref{eq:PoincarePolynomial_h=(4,5,5,8,10)}.
\begin{figure}[h]
\begin{center}
\begin{picture}(50,45)
\put(40,0){\framebox(40,15){$x_2-x_3$}}
\put(80,0){\framebox(40,15){$x_4-x_5$}}
\put(0,15){\framebox(40,15){$x_1-x_3$}}
\put(40,15){\framebox(40,15){$x_2-x_4$}}
\put(80,15){\framebox(40,15){$x_3-x_5$}}
\put(0,30){\framebox(40,15){$x_1-x_4$}}

\put(-70,4){{\footnotesize $\lambda_1^{I \setminus \Phi^+_\Theta}=2$}}
\put(-70,19){{\footnotesize $\lambda_2^{I \setminus \Phi^+_\Theta}=3$}}
\put(-70,34){{\footnotesize $\lambda_3^{I \setminus \Phi^+_\Theta}=1$}}
\end{picture}
\end{center}
\caption{The height distribution in $I \setminus \Phi^+_\Theta$.}
\label{pic:skew Young diagram(2,3,1)}
\end{figure}
\end{example}

\begin{corollary}
For any $\Theta$-ideal $I \subset \Phi^+$, we have 
\begin{align} \label{eq:Sommers-Tymoczko_Theta} 
\sum_{Y \in \mathcal{W}^{I, \Theta}} \q^{|Y|} = \prod_{i=1}^\rr (1+\q+\q^2+\cdots+\q^i)^{\mm_i^{I \setminus \Phi^+_\Theta}}, 
\end{align} 
where $\mathcal{W}^{I, \Theta}$ is defined in \eqref{eq:WeylTypeTheta} and $\mm_i^{I \setminus \Phi^+_\Theta}$ is defined in \eqref{eq:kiY} for $1 \leq i \leq \rr = \max\{\height(\alpha) \mid \alpha \in I \setminus \Phi^+_\Theta \}$. 
\end{corollary}

\begin{proof}
Combining Theorem~\ref{theorem:PoincarePolynomialHessGPWeylType} and Theorem~\ref{theorem:height_distributionHessTheta}, we obtain the desired equality. 
\end{proof}

\begin{remark}
Originally, the equality in \eqref{eq:Sommers-Tymoczko_Theta} for $\Theta = \emptyset$ was conjectured by Sommers and Tymoczko in \cite{SoTy} and they proved it for types $A,B,C, F_4, E_6$, and $G_2$  (\cite[Theorem~4.1]{SoTy}).
Also, Schauenburg confirmed \eqref{eq:Sommers-Tymoczko_Theta} in the case of $\Theta = \emptyset$ for types $D_5, D_6, D_7$, and $E_7$ by direct computation, and R\"{o}hrle obtained the result for type $D_4$ and $E_8$ (\cite[Theorem~1.28 and surrounding discussion]{Roh}).
For arbitrary Lie types, the equality \eqref{eq:Sommers-Tymoczko_Theta} for $\Theta = \emptyset$ was completely proved in \cite[Corollary~1.3]{AHMMS} by a classification-free argument.
\end{remark}

\subsection{Type $A$ formula} \label{subsection:typeAformula}

Let $N$ be a regular nilpotent matrix and $h: [n] \rightarrow [n]$ a Hessenberg function explained in Section~\ref{subsection:Type A description}.
Then the Poincar\'e polynomial for regular nilpotent Hessenberg varieties $\Hess(N,h)$ in type $A_{n-1}$ is described as 
\begin{align} \label{eq:PoincarePolynomialHess(N,h)}
\Poin(\Hess(N,h), \sqrt{\q}) = \prod_{j=1}^n (1+\q+\q^2+\cdots+\q^{h(j)-j}),
\end{align}
which immediately follows from \eqref{eq:PoincarePolynomialHessGBHeightDistribution}.
The aim of this section is to generalize \eqref{eq:PoincarePolynomialHess(N,h)} to the formula for regular nilpotent partial Hessenberg varieties $\Hess_\KK(N,h)$.

For $n \geq 1$, we define
\begin{align*}
[n]_\q \coloneqq 1+\q+\q^2+\cdots+\q^{n-1} \ \ \ \textrm{and} \ \ \ [n]_\q! \coloneqq \prod_{i=1}^n [i]_\q.
\end{align*}
Here, we take the convention $[0]_\q!=1$.
For $n \geq k \geq 0$, a \emph{$\q$-binomial coefficient} is defined as 
\begin{align*}
\begin{bmatrix}
\, n \, \\
\, k \, \\
\end{bmatrix}_\q \coloneqq \frac{[n]_\q!}{[k]_\q! [n-k]_\q!}. 
\end{align*}
Since $\q$-binomial coefficients satisfy the following recursive formula 
\begin{align*}
\begin{bmatrix}
\, n \, \\
\, k \, \\
\end{bmatrix}_\q 
= 
\begin{bmatrix}
\, n-1 \, \\
\, k \, \\
\end{bmatrix}_\q
+ 
\q^{n-k}\begin{bmatrix}
\, n-1 \, \\
\, k-1 \, \\
\end{bmatrix}_\q, 
\end{align*}
every $\q$-binomial coefficient $\begin{bmatrix}
\, n \, \\
\, k \, \\
\end{bmatrix}_\q $ is a polynomial in the variable $\q$ by induction on $n$. 
By using $\q$-binomial coefficients, we can describe the Poincar\'e polynomial of the regular nilpotent partial Hessenberg varieties in type $A$. 

\begin{theorem} \label{theorem:Poincare_polynomial_HessK(N,h)typeA}
Let $\KK = \{\kk_1,\ldots,\kk_s \}$ with $1 \leq \kk_1 < \cdots < \kk_s \leq n-1$. 
Let $N$ be a regular nilpotent matrix and $h: \KK \cup \{n \} \rightarrow \KK \cup \{n \}$ a partial Hessenberg function.
Then the Poincar\'e polynomial of the associated regular nilpotent partial Hessenberg variety $\Hess_\KK(N,h)$ in \eqref{eq:Hess(N,h)KKTypeA} is described as
\begin{align*}
\Poin(\Hess_\KK(N,h),\sqrt{\q}) = \prod_{j=1}^{s+1} \begin{bmatrix}
\, h(\kk_j)-\kk_{j-1} \, \\
\, \kk_j-\kk_{j-1} \, \\
\end{bmatrix}_\q
\end{align*}
where we take the convention that $\kk_0=0$ and $\kk_{s+1}=n$. 
\end{theorem}

\begin{proof}
Let $\tilde{h}: [n] \to [n]$ be the extended Hessenberg function in \eqref{eq:tilde h} for a partial Hessenberg function $h: \KK \cup \{n \} \rightarrow \KK \cup \{n \}$.
We define a partial Hessenberg function $h_\KK:\KK \cup \{n \} \rightarrow \KK \cup \{n \}$ by $h_\KK(\kk_j)=\kk_j$ for $1 \leq j \leq s+1$. 
Then by the discussion in the proof of Proposition~\ref{proposition:propertyHess(N,I)Theta}-(3), we have
\begin{align*}
\Poin(\Hess(N,\tilde{h}),\sqrt{\q}) = \Poin(\Hess(N_0,\widetilde{h_\KK}),\sqrt{\q}) \Poin(\Hess_\KK(N,h),\sqrt{\q}).
\end{align*}
Combining this with the formula in \eqref{eq:PoincarePolynomialHess(N,h)}, we conclude that 
\begin{align*}
\Poin(\Hess_\KK(N,h),\sqrt{\q}) =& \prod_{i=1}^n \frac{1+\q+\q^2+\cdots+\q^{\tilde{h}(i)-i}}{1+\q+\q^2+\cdots+\q^{\widetilde{h_\KK}(i)-i}} \\
=& \prod_{j=1}^{s+1} \left( \prod_{i=\kk_{j-1}+1}^{\kk_{j}} \frac{1+\q+\q^2+\cdots+\q^{\tilde{h}(i)-i}}{1+\q+\q^2+\cdots+\q^{\widetilde{h_\KK}(i)-i}} \right) \\
=& \prod_{j=1}^{s+1} \frac{[h(\kk_j)-\kk_{j-1}]_\q [h(\kk_j)-\kk_{j-1}-1]_\q \cdots [h(\kk_j)-\kk_{j}+1]_\q}{[\kk_j-\kk_{j-1}]_\q!} \\
=& \prod_{j=1}^{s+1} \begin{bmatrix}
\, h(\kk_j)-\kk_{j-1} \, \\
\, \kk_j-\kk_{j-1} \, \\
\end{bmatrix}_\q,
\end{align*}
as desired.
\end{proof}

\begin{example}
Consider the case of type $A_9$ and take $\KK=\{2,4,5,8\}$ and $h=(4,5,5,8,10)$, which is discussed in Example~\ref{example:h=(4,5,5,8,10)PoincarePolynomial}.
By Theorem~\ref{theorem:Poincare_polynomial_HessK(N,h)typeA} we have 
\begin{align*}
\Poin(\Hess_\KK(N,h),\sqrt{\q}) &= \begin{bmatrix}
\, 4 \, \\
\, 2 \, \\
\end{bmatrix}_\q \cdot \begin{bmatrix}
\, 3 \, \\
\, 2 \, \\
\end{bmatrix}_\q \cdot \begin{bmatrix}
\, 1 \, \\
\, 1 \, \\
\end{bmatrix}_\q \cdot \begin{bmatrix}
\, 3 \, \\
\, 3 \, \\
\end{bmatrix}_\q \cdot \begin{bmatrix}
\, 2 \, \\
\, 2 \, \\
\end{bmatrix}_\q = \frac{[4]_\q[3]_\q}{[2]_\q[1]_\q} \cdot [3]_\q \\ 
&= \frac{(1+\q+\q^2+\q^3)(1+\q+\q^2)}{1+\q}(1+\q+\q^2) = (1+\q^2)(1+\q+\q^2)^2. 
\end{align*}
\end{example}

\bigskip

\section{Cohomology ring $H^*(\Hess_\Theta(N,I))$} \label{sect:cohomology}

It is known that the cohomology ring of $G/P$ is isomorphic to the invariants in the cohomology ring of $G/B$ under the action of $W_\Theta$ by \cite{BGG}.
In this section we generalize this fact to regular nilpotent partial Hessenberg varieties.

We first explain Borel's work in \cite{Bor53}. 
We also refer the reader to \cite{BGG}. 
In what follows, we may identify the character group $\Hom(T,\C^*)$ of $T$ with a lattice $\t^*_{\Z}$ through differential at the identity element of $T$. 
We define $\t^*_{\Q} \coloneqq \t^*_{\Z} \otimes_{\Z} \Q$ and its symmetric algebra is denoted by
\begin{align*}
\RR \coloneqq \Sym \t^*_{\Q}.
\end{align*}
The Weyl group $W$ acts on $\t^*_\Q$ by the formula 
\begin{align} \label{eq:reflection}
s_\beta(\alpha) = \alpha - \frac{2(\alpha,\beta)}{(\beta,\beta)}\beta \ \ \ \textrm{for} \ \alpha, \beta \in \Phi 
\end{align}
where $( \ , \ )$ is the $W$-invariant nondegenerate positive-definite bilinear symmetric form on $\t^*_\Q$ transferred from the Killing form of $\g$ restricted to $\t$.
We also recall that $\frac{2(\alpha,\beta)}{(\beta,\beta)}$ is an integer for $\alpha, \beta \in \Phi $.
The $W$-action on $\t^*_{\Q}$ naturally extends a $W$-action on $\RR$. 
Since $B$ is a semidirect product $T \ltimes U$, each $\alpha \in \Hom(T,\C^*)$ extends to a homomorphism $\tilde\alpha: B \rightarrow \C^*$.
In fact, for any $b \in B$, there are unique $t \in T$ and $u \in U$ such that $b=tu$.
Then we define $\tilde\alpha(b) = \alpha(t)$. 
It is easy to check that $\tilde\alpha$ is a homomorphism. 
We denote by $\C_{\tilde \alpha}$ the one-dimensional $B$-module via $\tilde\alpha: B \rightarrow \C^*$.  
We define the complex line bundle $L_\alpha \coloneqq G \times_B \C_{\tilde \alpha}$ over the flag variety $G/B$ where $G \times_B \C_{\tilde \alpha}$ is the quotient of the direct product $G \times \C_{\tilde \alpha}$ by the left $B$-action given by $b \cdot (g,z) = (gb^{-1},\tilde{\alpha}(b) z)$ for $b \in B$ and $(g,z) \in G \times \C_{\tilde \alpha}$. 
To each $\alpha \in \t^*_\Z$, we assign the first Chern class $c_1(L_{\alpha}^*)$ of the dual line bundle $L_{\alpha}^*$. 
This yields the following ring homomorphism
\begin{align} \label{eq:homomorphism_flag}
\varphi: \RR \rightarrow H^*(G/B); \ \ \ \alpha \mapsto c_1(L_{\alpha}^*) 
\end{align} 
which doubles the grading on $\RR$.
Recall that the Weyl group $W$ acts on $H^*(G/B)$. 
In fact, let $K \subset G$ be a maximal compact subgroup such that $T_K \coloneqq K \cap T$ is a maximal torus in $K$. Then the natural mapping $K/T_K \rightarrow G/B$ gives a homeomorphism. 
It is also known that the Weyl group $W = N_G(T)/T$ is isomorphic to $N_K(T_K)/T_K$. 
Since $N_K(T_K)/T_K$ acts on $K/T_K$ by $(gT_K) \cdot (zT_K) = gzT_K$ for $gT_K \in K/T_K$ and $zT_K \in N_K(T_K)/T_K$, we obtain an $W$-action on $G/B$ via the homeomorphism $G/B \approx K/T_K$ and the isomorphism $W \cong N_K(T_K)/T_K$, which induces the $W$-action on $H^*(G/B)$.
As is well-known, the map $\varphi$ in \eqref{eq:homomorphism_flag} is a surjective $W$-equivariant map and its kernel is the ideal $(\RR^W_+)$ generated by the $W$-invariants in $\RR$ with zero constant term by Borel's theorem (\cite{Bor53}). 
In particular, this induces the following isomorphism of graded $\Q$-algebras
\begin{align*} 
H^*(G/B) \cong \RR/(\RR^W_+).  
\end{align*} 

\begin{remark} \label{remark:Weyl_group}
In order to see a $W$-action on $H^*(G/B)$ above, we used an isomorphism $W = N_G(T)/T \cong N_K(T_K)/T_K$. 
A correspondence of the isomorphism $N_K(T_K)/T_K \cong N_G(T)/T$ is given by $zT_K \mapsto zT$. 
In fact, it is clear that $(K/T_K)^{T_K} = N_K(T_K)/T_K$. 
On the other hand, one can see that $W=N_G(T)/T$ acts freely and transitively on $(G/B)^T$ (e.g. \cite[Section~24.1]{Hum75}), so
we have the identification  $(G/B)^T \cong N_G(T)/T$ which sends $zB$ to $zT$. 
Since the homeomorphism $K/T_K \approx G/B$ is a $T_K$-equivariant map, this induces $(K/T_K)^{T_K} \cong (G/B)^{T_K}$. 
Here, we note that $(G/B)^{T_K} = (G/B)^T$ since $(G/B)^{T_K} \supset (G/B)^T$ and they have the same cardinality $|W|$. 
Hence, we obtain the isomorphism $N_K(T_K)/T_K = (K/T_K)^{T_K} \cong (G/B)^{T_K} = (G/B)^T \cong N_G(T)/T$ which sends $zT_K$ to $zT$.
\end{remark}

The following theorem is well-known.

\begin{theorem} $($\cite[Corollary~5.4 and Theorem~5.5]{BGG}$)$ \label{theorem:BGG}
The homomorphism $\pi^*: H^*(G/P) \rightarrow H^*(G/B)$ induced from the natural projection
$\pi: G/B \to G/P$ is injective and its image coincides with $H^*(G/B)^{W_\Theta}$, which is the invariants in $H^*(G/B)$ under the action of $W_\Theta$. 
In particular, $\pi^*: H^*(G/P) \hookrightarrow H^*(G/B)$ induces the isomorphism
\begin{align} \label{eq:BGG}
H^*(G/P) \cong H^*(G/B)^{W_\Theta}
\end{align} 
as graded $\Q$-algebras. 
\end{theorem}

\begin{lemma} \label{lemma:WTheta-action on X(x,H)}
For $x \in \g$ and a $\p$-Hessenberg space $H \subset \g$, we set
\begin{align*}
X(x,H) \coloneqq \{gT_K \in K/T_K \mid \Ad(g^{-1})(x) \in H \}.
\end{align*}
Note that $\Ad$ means the restriction of the adjoint map $\Ad: G \to \Aut(\g)$ to $K$. 
Then the $W_\Theta$-action on $K/T_K$ preserves $X(x,H)$. 
In particular, since the natural mapping $X(x,H) \rightarrow \Hess(x,H)$ is a homeomorphism, the $W_\Theta$-action on $G/B$ preserves $\Hess(x,H)$.
In other words, the inclusion map $\Hess(x,H) \subset G/B$ is a $W_\Theta$-equivariant map.
\end{lemma}

\begin{proof}
Let $gT_K \in X(x,H)$ and $w \in W_\Theta$. 
Since $W_\Theta$ can be regarded as a subgroup of $N_K(T_K)/T_K$ under the isomorphism $N_K(T_K)/T_K \cong N_G(T)/T=W$, we write $w=zT_K$ for some $z \in N_K(T_K)$. 
Then we show that $gzT_K \in X(x,H)$. 
In other words, it suffices to show that $\Ad((gz)^{-1})(x) \in H$. 
We have $\Ad((gz)^{-1})(x) = \Ad(z^{-1}) (\Ad(g^{-1})(x)) \in \Ad(z^{-1})(H)$ by $gT_K \in X(x,H)$, so it is enough to show that $z \in P$. 
In fact, if $z \in P$, then we have $\Ad(z^{-1})(H) \subset H$ since $H$ is a $\p$-Hessenberg space. 
We now prove that $z \in P$.
By Remark~\ref{remark:Weyl_group}, the isomorphism $N_K(T_K)/T_K \cong N_G(T)/T=W$ sends $zT_K$ to $zT$. 
Thus, one can write $w=zT \in W_\Theta$ in $W = N_G(T)/T$.   
Since $P=BW_\Theta B$, we have $z \in P$.
Therefore, we conclude that $\Ad((gz)^{-1})(x) \in H$ and hence we have $(gT_K) \cdot w = gzT_K \in X(x,H)$. 
\end{proof}

The main theorem is to generalize Theorem~\ref{theorem:BGG} to regular nilpotent partial Hessenberg varieties as follows.

\begin{theorem} \label{theorem:main}
Let $N$ be a regular nilpotent element in $\g$ and $I$ a $\Theta$-ideal in $\Phi^+$.
Let $\Hess(N,I) \subset G/B$ and $\Hess_\Theta(N,I) \subset G/P$ be the associated regular nilpotent Hessenberg variety in \eqref{eq:Hess(N,I)} and the associated regular nilpotent partial Hessenberg variety in \eqref{eq:Hess(N,I)Theta}, respectively. 
Then, the homomorphism $\pi_I^*: H^*(\Hess_\Theta(N,I)) \rightarrow H^*(\Hess(N,I))$ induced from the natural projection $\pi_I: \Hess(N,I) \to \Hess_\Theta(N,I)$ is injective and its image coincides with $H^*(\Hess(N,I))^{W_\Theta}$, the invariants in $H^*(\Hess(N,I))$ under the action of $W_\Theta$ induced from Lemma~\ref{lemma:WTheta-action on X(x,H)}. 
In other words, $\pi_I^*: H^*(\Hess_\Theta(N,I)) \hookrightarrow H^*(\Hess(N,I))$ yields the isomorphism of graded $\Q$-algebras
\begin{align*}
H^*(\Hess_\Theta(N,I)) \cong H^*(\Hess(N,I))^{W_\Theta}.
\end{align*} 
\end{theorem}

In order to prove Theorem~\ref{theorem:main}, we consider the following commutative diagram, which is discussed in the proof of Proposition~\ref{proposition:propertyHess(N,I)Theta}-(3)
\begin{center}
\begin{tikzcd} 
H^*(G/P) \arrow[r, "\pi^*"] \arrow[rightarrow,d, "j_{I, \Theta}^*"'] &[0.5em] H^*(G/B) \arrow[rightarrow,d, "j_I^*"'] \arrow[r, "\iota^*"] &[0.5em] H^*(P/B) \\
H^*(\Hess_\Theta(N,I))  \arrow[r, "\pi_I^*"] &[0.5em]  H^*(\Hess(N,I)) \arrow[r, "\iota_I^*"]  & H^*(P/B) \arrow[u,equal]  
\end{tikzcd}
\end{center}
where the horizontal arrows are induced from the fiber bundles by Lemma~\ref{lemma:fiber bundle} and the vertical arrows are induced from the inclusion maps $j_I: \Hess(N,I) \hookrightarrow G/B$ and $j_{I,\Theta}: \Hess_\Theta(N,I) \hookrightarrow G/P$ respectively. 

\begin{theorem} $($\cite[Theorem~1.1]{AHMMS}$)$ \label{theorem:restriction_map_surjective_AHMMS}
Let $I \subset \Phi^+$ be a lower ideal. 
Then, the restriction map $j_I^*: H^*(G/B) \to H^*(\Hess(N,I))$ is surjective. 
\end{theorem}

\begin{proposition} \label{proposition:main}
Let $I \subset \Phi^+$ be a $\Theta$-ideal. 
Then the following holds.
\begin{enumerate}
\item[(1)] The restriction map $j_{I,\Theta}^*: H^*(G/P) \to H^*(\Hess_\Theta(N,I))$ is surjective.
\item[(2)] The map $\pi_I^*: H^*(\Hess_\Theta(N,I)) \rightarrow H^*(\Hess(N,I))$ is injective. 
\end{enumerate}
\end{proposition}

\begin{proof}
(1) As discussed in the proof of Proposition~\ref{proposition:propertyHess(N,I)Theta}-(3), both $\iota^*:H^*(G/B) \to H^*(P/B)$ and $\iota_I^*: H^*(\Hess(N,I)) \to H^*(P/B)$ are surjective. 
Let $s: H^*(P/B) \to H^*(G/B)$ be a section of $\iota^*$, namaly $\iota^* \circ s = \id$.  
Then the Leray--Hirsch theorem yields the following isomorphism of $\Q$-vector spaces 
\begin{align*}
\rho: H^*(P/B) \otimes_\Q H^*(G/P) \cong H^*(G/B),
\end{align*}
which sends $\sum \alpha \otimes \beta$ to $\sum s(\alpha) \cdot \pi^*(\beta)$. 
Define $s_I: H^*(P/B) \to H^*(\Hess(N,I))$ by $s_I \coloneqq j_I^* \circ s$. 
Then one can easily see that $s_I$ is a section of $\iota_I^*: H^*(\Hess(N,I)) \to H^*(P/B)$.
By the Leray--Hirsch isomorphism we have the isomorphism of $\Q$-vector spaces as folllows
\begin{align} \label{eq:rhoI}
\rho_I: H^*(P/B) \otimes_\Q H^*(\Hess_\Theta(N,I)) \cong H^*(\Hess(N,I)),
\end{align}
which sends $\sum \alpha \otimes \beta$ to $\sum s_I(\alpha) \cdot \pi_I^*(\beta)$. 
Hence, we obtain the following commutative diagram
\begin{center}
\begin{tikzcd}
H^*(P/B) \otimes_\Q H^*(G/P) \arrow[r, "\rho", "\cong"'] \arrow[rightarrow,d, "\id \otimes j_{I, \Theta}^*"'] &[0.5em] H^*(G/B) \arrow[rightarrow,d, "j_I^*"'] \\
H^*(P/B) \otimes_\Q H^*(\Hess_\Theta(N,I)) \arrow[r, "\rho_I", "\cong"'] &[0.5em]  H^*(\Hess(N,I)). 
\end{tikzcd}
\end{center}
Since the restriction map $j_I^*: H^*(G/B) \to H^*(\Hess(N,I))$ is surjective by Theorem~\ref{theorem:restriction_map_surjective_AHMMS}, the map $\id \otimes j_{I, \Theta}^*: H^*(P/B) \otimes_\Q H^*(G/P) \to H^*(P/B) \otimes_\Q H^*(\Hess_\Theta(N,I))$ is also surjective, which implies the surjectivity of $j_{I,\Theta}^*: H^*(G/P) \to H^*(\Hess_\Theta(N,I))$.
In fact, consider the surjective map $\xi: H^*(P/B) \otimes_\Q H^*(G/P) \to H^*(G/P)$ defined by $\xi(\sum_{i+j=k} \alpha_i \otimes \beta_j)=\alpha_0 \cdot \beta_k$ in each degree $k$ where we note that $\alpha_0 \in H^0(P/B) \cong \Q$.
By the similar way we define the surjective map $\xi_I: H^*(P/B) \otimes_\Q H^*(\Hess_\Theta(N,I)) \to H^*(\Hess_\Theta(N,I))$ and they make the following commutative diagram:
\begin{center}
\begin{tikzcd}
H^*(P/B) \otimes_\Q H^*(G/P) \arrow[r, "\xi"] \arrow[rightarrow,d, "\id \otimes j_{I, \Theta}^*"'] &[0.5em] H^*(G/P) \arrow[rightarrow,d, "j_{I, \Theta}^*"'] \\
H^*(P/B) \otimes_\Q H^*(\Hess_\Theta(N,I)) \arrow[r, "\xi_I"] &[0.5em]  H^*(\Hess_\Theta(N,I)). 
\end{tikzcd}
\end{center}
Since both $\id \otimes j_{I, \Theta}^*$ and $\xi_I$ are surjective, $j_{I, \Theta}^*$ is also surjective.

(2) Define the map $\zeta_I: H^*(\Hess_\Theta(N,I)) \to H^*(P/B) \otimes_\Q H^*(\Hess_\Theta(N,I))$ by $\zeta_I(x) = 1 \otimes x$, which is injective since $H^0(P/B) \cong \Q$.
This makes the following commutative diagram:
\begin{center}
\begin{tikzcd}
H^*(P/B) \otimes_\Q H^*(\Hess_\Theta(N,I)) \arrow[r, "\rho_I", "\cong"'] &[0.5em] H^*(\Hess(N,I)) \\
 &[0.5em]  H^*(\Hess_\Theta(N,I)) \arrow[hookrightarrow,lu, "\zeta_I"] \arrow[rightarrow,u, "\pi_I^*"'] 
\end{tikzcd}
\end{center}
where $\rho_I$ is the Leray--Hirsch isomorphism defined in \eqref{eq:rhoI}. 
Since $\zeta_I$ is injective, $\pi_I^*$ is also injective.
\end{proof}

\begin{proof}[Proof of Theorem~\ref{theorem:main}]
The injectivity of $\pi_I^*: H^*(\Hess_\Theta(N,I)) \rightarrow H^*(\Hess(N,I))$ follows from Proposition~\ref{proposition:main}-(2), so we prove that the image of $\pi_I^*$ equals $H^*(\Hess(N,I))^{W_\Theta}$. 
Since the restriction map $j_I^*: H^*(G/B) \rightarrow H^*(\Hess(N,I))$ is a $W_\Theta$-equivariant map by Lemma~\ref{lemma:WTheta-action on X(x,H)}, this induces the map $(j_I^*)^{W_\Theta}: H^*(G/B)^{W_\Theta} \rightarrow H^*(\Hess(N,I))^{W_\Theta}$.
One can see that $(j_I^*)^{W_\Theta}$ is a surjective map since $j_I^*$ is surjective. 
In fact, for any $y \in H^*(\Hess(N,I))^{W_\Theta}$, there exists $x \in H^*(G/B)$ such that $j_I^*(x)=y$. 
Putting $x'= \frac{1}{|W_\Theta|}\sum_{w \in W_\Theta} w \cdot x \in H^*(G/B)^{W_\Theta}$, we have $(j_I^*)^{W_\Theta}(x')=j_I^*(x')=y$. 
Consider the following commutative 
\begin{center}
\begin{tikzcd} 
H^*(G/P) \arrow[hookrightarrow,r, "\pi^*"] \arrow[twoheadrightarrow,d, "j_{I, \Theta}^*"'] &[0.5em] H^*(G/B) \arrow[twoheadrightarrow,d, "j_I^*"'] \\
H^*(\Hess_\Theta(N,I))  \arrow[hookrightarrow,r, "\pi_I^*"] &[0.5em]  H^*(\Hess(N,I)) 
\end{tikzcd}
\end{center}
where the surjectivity of $j_I^*$ and $j_{I, \Theta}^*$ follow from Theorem~\ref{theorem:restriction_map_surjective_AHMMS} and Proposition~\ref{proposition:main}-(1). 
Since the image of $\pi^*$ coincides with $H^*(G/B)^{W_\Theta}$ by Theorem~\ref{theorem:BGG}, the image of $\pi_I^*$ is included in $H^*(\Hess(N,I))^{W_\Theta}$ by the surjectivity of $j_{I, \Theta}^*$.
Hence, the commutative diagram above yields the following commutative diagram
\begin{center}
\begin{tikzcd} 
H^*(G/P) \arrow[r, "\cong"] \arrow[twoheadrightarrow,d, "j_{I, \Theta}^*"'] &[0.5em] H^*(G/B)^{W_\Theta} \arrow[twoheadrightarrow,d, "(j_I^*)^{W_\Theta}"'] \\
H^*(\Hess_\Theta(N,I))  \arrow[hookrightarrow,r, ] &[0.5em]  H^*(\Hess(N,I))^{W_\Theta}. 
\end{tikzcd}
\end{center}
Since the bottom arrow is surjective, the homomorphism $\pi_I^*$ yields the isomorphism $H^*(\Hess_\Theta(N,I)) \cong H^*(\Hess(N,I))^{W_\Theta}$ as desired.
\end{proof}

\bigskip

\section{Invariants in logarithmic derivation modules} \label{sect:invariants logarithmic derivation modules}

By the work of \cite{AHMMS}, we can describe the cohomology rings of regular nilpotent Hessenberg varieties by the logarithmic derivation modules of ideal arrangements. 
We generalize this result to regular nilpotent partial Hessenberg varieties in this section.

\subsection{Logarithmic derivation module} \label{subsect:logarithmic derivation module}
A \emph{derivation} of $\RR=\Sym \t^*_{\Q}$ over $\Q$ is a $\Q$-linear map $\der: \RR \to \RR$ such that 
\begin{align} \label{eq:derivation}
\der(fg)= \der(f)g + f\der(g) \ \textrm{for} \ f,g \in \RR.
\end{align}
The \emph{derivation module} $\Der \RR$ of $\RR$ is the collection of all derivations of $\RR$ over $\Q$. 
For a lower ideal $I \subset \Phi^+$, the \emph{ideal arrangement} $\AA_I$ is the set of  hyperplanes orthogonal to $\alpha \in I$.
Its \emph{logarithmic derivation module} $D(\AA_I)$ is defined to be the following $\RR$-module:
\begin{align*}
D(\AA_I) \coloneqq \{\der \in \Der \RR \mid \der(\alpha) \in \RR \, \alpha \ \textrm{for all} \ \alpha \in I \}.
\end{align*}
It follows from \cite[Theorem~1.1]{ABCHT} that $D(\AA_I)$ is a free $\RR$-module for all lower ideal $I \subset \Phi^+$. 
We remark that \cite{ABCHT} also gives the exponents of $\AA_I$ as the dual partition of the height distribution in $I$ where the exponents of $\AA_I$ denotes a sequence of degrees for a homogeneous basis of $D(\AA_I)$. 
Let $Q \in \Sym^2 (\t^*_\Q)^W$ be a $W$-invariant non-degenerate quadratic form.
For a lower ideal $I \subset \Phi^+$, we define an ideal $\a(I)$ of $\RR$ by 
\begin{align} \label{eq:a(I)}
\a(I) \coloneqq \{\der(Q) \in \RR \mid \der \in D(\AA_I) \}.
\end{align}
Note that the ideal $\a(I)$ is independent of a choice of $Q$ (see \cite[Remark~3.6]{AHMMS}).
For a lower ideal $I \subset \Phi^+$, we denote the composition of $\varphi$ in \eqref{eq:homomorphism_flag} and the restriction map induced from the inclusion $j_I: \Hess(N,I) \hookrightarrow G/B$ by
\begin{align} \label{eq:homomorphism_Hess(N,I)}
\varphi_I: \RR \xrightarrow{\varphi} H^*(G/B) \xrightarrow{j_I^*} H^*(\Hess(N,I)).
\end{align} 
Note that both $\varphi$ and $j_I^*$ are surjective maps by \cite{Bor53} and Theorem~\ref{theorem:restriction_map_surjective_AHMMS}, so the map $\varphi_I$ is also surjective. 

\begin{theorem} $($\cite[Theorem~1.1]{AHMMS}$)$ \label{theorem:AHMMS}
Let $I \subset \Phi^+$ be a lower ideal. 
Then, the kernel of the surjective map $\varphi_I$ in \eqref{eq:homomorphism_Hess(N,I)} coincides with $\a(I)$ defined in \eqref{eq:a(I)}.
In particular, $\varphi_I$ yields the isomorphism of graded $\Q$-algebras
\begin{align*} 
H^*(\Hess(N,I)) \cong \RR/\a(I).  
\end{align*} 
\end{theorem}

\subsection{$W_\Theta$-action}
The aim of this section is to see that $D(\AA_I)$ admits an $W_\Theta$-action for arbitrary $\Theta$-ideal $I$.
The $W$-action on $\RR =\Sym \t^*_\Q$ induces a $W$-action on the derivation module $\Der \RR$.
In fact, for $w \in W$ and $\der \in \Der \RR$, the derivation $w \cdot \der$ is defined by 
\begin{align} \label{eq:W-action_derivation}
(w \cdot \der) (f) \coloneqq w(\der(w^{-1}f)) \ \ \ \textrm{for}  \ f \in \RR.
\end{align} 
Then one can easily see that $w \cdot \der$ satisfies \eqref{eq:derivation}, and hence we have $w \cdot \der \in \Der \RR$. 

\begin{lemma} \label{lemma:WTheta-action}
Let $I \subset \Phi^+$ be a $\Theta$-ideal. 
If $\alpha \in I$ and $\beta \in \Theta$ with $\alpha \neq \beta$, then $s_\beta(\alpha) \in I$.
\end{lemma}

\begin{proof}
Since $s_\beta(\Phi^+ \setminus \{\beta \}) = \Phi^+ \setminus \{\beta \}$ (cf. \cite[Proposition~1.4]{Hum90}), we see that $s_\beta(\alpha) \in \Phi^+$.
By the formula in \eqref{eq:reflection} we have $s_\beta(\alpha) - \alpha \in \Z \beta$, so either $\alpha \preceq_\Theta s_\beta(\alpha)$ or $s_\beta(\alpha) \preceq_\Theta \alpha$ holds. 
If $\alpha \preceq_\Theta s_\beta(\alpha)$, then $s_\beta(\alpha) \in I$ from the condition \eqref{eq:Theta upper ideal}. 
If $s_\beta(\alpha) \preceq_\Theta \alpha$, then we have in particular $s_\beta(\alpha) \preceq \alpha$, so $s_\beta(\alpha) \in I$ by the condition \eqref{eq:lower ideal}.
\end{proof}

\begin{proposition} \label{proposition:WTheta-action}
Let $I \subset \Phi^+$ be a $\Theta$-ideal. Then the following holds.
\begin{enumerate}
\item[(1)] The $W_\Theta$-action on $\Der \RR$ preserves the logarithmic derivation module $D(\AA_I)$. 
\item[(2)] The $W_\Theta$-action on $\RR$ induces the $W_\Theta$-action on the quotient ring $\RR/\a(I)$.
\end{enumerate}
\end{proposition}

\begin{proof}
(1) Let $\der \in D(\AA_I)$ and $\beta \in \Theta$. 
Then it suffices to show that $s_\beta \cdot \der \in D(\AA_I)$ since $W_\Theta$ is generated by $s_\beta$'s for $\beta \in \Theta$. 
Take a root $\alpha \in I$.
If $\alpha \neq \beta$, then we have 
\begin{align*} 
(s_\beta \cdot \der)(\alpha) = s_\beta(\der(s_\beta^{-1}(\alpha))) = s_\beta(\der(s_\beta (\alpha)).
\end{align*}
It follows from Lemma~\ref{lemma:WTheta-action} that $s_\beta(\alpha) \in I$, so one has $\der(s_\beta (\alpha)) \in \RR \, s_\beta (\alpha)$ since $\der \in D(\AA_I)$. 
This implies that $(s_\beta \cdot \der)(\alpha) = s_\beta(\der(s_\beta (\alpha)) \in \RR \, \alpha$ for all $\alpha \in I$ with $\alpha \neq \beta$.
If $\alpha = \beta$, then 
\begin{align*} 
(s_\beta \cdot \der)(\beta) = s_\beta(\der(s_\beta^{-1}(\beta))) = -s_\beta(\der(\beta)) \in \RR \, \beta
\end{align*}
since $\der \in D(\AA_I)$ and $\beta \in \Theta \subset I$. 
Therefore, we conclude that $s_\beta \cdot \der \in D(\AA_I)$ as desired. 

(2) For $w \in W_\Theta$ and $f \in \a(I)$, we show that $w(f) \in \a(I)$.
By the definition of $\a(I)$, we can write $f=\der(Q)$ for some $\der \in D(\AA_I)$. 
Then we have 
\begin{align*} 
w(f) = w(\der(Q)) = w(\der(w^{-1}Q)) = (w \cdot \der)(Q) 
\end{align*}
where the second equality above follows from the definition that $Q$ is $W$-invariant. 
We obtain $w \cdot \der \in D(\AA_I)$ from (1), so $w(f) = (w \cdot \der)(Q) \in \a(I)$. 
Hence, the $W_\Theta$-action on $\RR$ induces the $W_\Theta$-action on the quotient ring $\RR/\a(I)$.
\end{proof}

\begin{remark}
Let $N$ be a regular nilpotent element in $\g$ and $I$ a $\Theta$-ideal in $\Phi^+$.
The $W_\Theta$-action on $G/B$ preserves $\Hess(N,I)$ from Lemma~\ref{lemma:WTheta-action on X(x,H)}. 
This induces the $W_\Theta$-action on $H^*(\Hess(N,I))$.
On the other hand, the $W_\Theta$-action on $\RR$ induces the $W_\Theta$-action on the quotient ring $\RR/\a(I)$ by Proposition~\ref{proposition:WTheta-action}-(2).
Then it is straightforward to see that $H^*(\Hess(N,I)) \cong \RR/\a(I)$ in Theorem~\ref{theorem:AHMMS} is an isomorphism as $\Q[W_\Theta]$-modules.
\end{remark}

\subsection{$W_\Theta$-invariants}
Let $I$ be a $\Theta$-ideal in $\Phi^+$ and consider the homomorphism $\varphi_I: \RR \xrightarrow{\varphi} H^*(G/B) \xrightarrow{j_I^*} H^*(\Hess(N,I))$ defined in \eqref{eq:homomorphism_Hess(N,I)}.
Recall that $\varphi: \RR \to H^*(G/B)$ is a surjective $W_\Theta$-equivariant map by \cite{Bor53}. 
It also follows from Theorem~\ref{theorem:restriction_map_surjective_AHMMS} and Lemma~\ref{lemma:WTheta-action on X(x,H)} that the restriction map $j_I^*: H^*(G/B) \to H^*(\Hess(N,I))$ is surjective and $W_\Theta$-equivariant. 
As discussed in the proof of Theorem~\ref{theorem:main}, the maps $\varphi$ and $j_I^*$ induce the surjective maps $\varphi^{W_\Theta}: \RR^{W_\Theta} \twoheadrightarrow H^*(G/B)^{W_\Theta}$ and $(j_I^*)^{W_\Theta}: H^*(G/B)^{W_\Theta} \twoheadrightarrow H^*(\Hess(N,I))^{W_\Theta}$. 
Therefore, we conclude from Theorems~\ref{theorem:BGG} and \ref{theorem:main} that the composition map $\varphi_I^{W_\Theta}: \RR^{W_\Theta} \xrightarrow{\varphi^{W_\Theta}} H^*(G/B)^{W_\Theta} \xrightarrow{(j_I^*)^{W_\Theta}} H^*(\Hess(N,I))^{W_\Theta}$ yields the following surjective homomorphism 
\begin{align} \label{eq:homomorphism_Hess(N,I)Theta}
\varphi_{I, \Theta}: \RR^{W_\Theta} \twoheadrightarrow H^*(G/P) \twoheadrightarrow H^*(\Hess_\Theta(N,I))
\end{align}
where the first map is the composition of the isomorphism \eqref{eq:BGG} and $\varphi^{W_\Theta}$, and the second map is the restriction map induced from the inclusion $j_{I,\Theta}: \Hess_\Theta(N,I) \hookrightarrow G/P$. 
Our goal is to describe the kernel of $\varphi_{I, \Theta}$ in terms of $W_\Theta$-invariants in logarithmic derivation modules. 

Recall that the derivation module $\Der \RR$ of $\RR=\Sym \t^*_{\Q}$ admits the $W$-action by \eqref{eq:W-action_derivation}. 
We set
\begin{align*}
(\Der \RR)^{W_\Theta} \coloneqq \{\der \in \Der \RR \mid w \cdot \der = \der \ \textrm{for all} \ w \in W_\Theta \}.
\end{align*}
It is clear that $(\Der \RR)^{W_\Theta}$ is an $\RR^{W_\Theta}$-module. 
It is known that $(\Der \RR)^{W_\Theta}$ is a free $\RR^{W_\Theta}$-module of rank $n$ (cf. \cite[Lemma~6.48]{OrTe}).
If $I$ is a $\Theta$-ideal in $\Phi^+$, then the $W_\Theta$-action on $\Der \RR$ preserves the logarithmic derivation module $D(\AA_I)$ by Proposition~\ref{proposition:WTheta-action}-(1).
For a $\Theta$-ideal $I \subset \Phi^+$, we define an $\RR^{W_\Theta}$-submodule of $(\Der \RR)^{W_\Theta}$ by 
\begin{align*}
D(\AA_I)^{W_\Theta} \coloneqq D(\AA_I) \cap (\Der \RR)^{W_\Theta} = \{\der \in (\Der \RR)^{W_\Theta} \mid \der(\alpha) \in \RR \, \alpha \ \textrm{for all} \ \alpha \in I \}.
\end{align*}
Fix a $W$-invariant non-degenerate quadratic form $Q \in \Sym^2 (\t^*_\Q)^W$.
Then one can see that $\der(Q)$ belongs to $\RR^{W_\Theta}$ for any $\der \in (\Der \RR)^{W_\Theta}$.
In fact, for any $w \in W_\Theta$, we have 
\begin{align} \label{eq:der(Q)invariant}
w(\der(Q)) = w(\der(w^{-1} Q)) = (w \cdot \der)(Q) = \der(Q).
\end{align} 
For a $\Theta$-ideal $I \subset \Phi^+$, we define an ideal $\a(I)_\Theta$ of $\RR^{W_\Theta}$ by 
\begin{align} \label{eq:a(I)Theta} 
\a(I)_\Theta \coloneqq \{\der(Q) \in \RR^{W_\Theta} \mid \der \in D(\AA_I)^{W_\Theta} \}.\end{align}
By a similar argument of \cite[Remark~3.6]{AHMMS}, the ideal $\a(I)_\Theta$ does not dependend on a choice of $Q$. 

\begin{theorem}
Let $I$ be a $\Theta$-ideal in $\Phi^+$. The kernel of the surjective map $\varphi_{I, \Theta}$ in  \eqref{eq:homomorphism_Hess(N,I)Theta} coincides with the ideal $\a(I)_\Theta$ in \eqref{eq:a(I)Theta}. 
In particular, $\varphi_{I,\Theta}$ induces the isomorphism
\begin{align*} 
H^*(\Hess_\Theta(N,I)) \cong \RR^{W_\Theta}/\a(I)_\Theta 
\end{align*} 
as graded $\Q$-algebras. 
\end{theorem}

\begin{proof}
By the definition of $\varphi_{I, \Theta}$, we have the following commutative diagram:
\begin{center}
\begin{tikzcd} 
\varphi_I: &[-3em] \RR \arrow[twoheadrightarrow,r, "\varphi"] &[0.5em] H^*(G/B) \arrow[twoheadrightarrow,r, "j_I^*"] &[0.5em] H^*(\Hess(N,I)) \\
\varphi_{I, \Theta}: &[-3em] \RR^{W_\Theta}  \arrow[twoheadrightarrow,r, ""] \arrow[hookrightarrow,u, "{\rm inclusion}"'] &[0.5em]  H^*(G/P) \arrow[twoheadrightarrow,r, "j_{I, \Theta}^*"] \arrow[hookrightarrow,u, "\pi^*"'] & H^*(\Hess_\Theta(N,I)). \arrow[hookrightarrow,u, "\pi_I^*"']
\end{tikzcd}
\end{center}
It follows from Theorem~\ref{theorem:AHMMS} that $\ker \varphi_I = \a(I)$. 
Since $\ker \varphi_{I, \Theta} = \ker \varphi_I \cap \RR^{W_\Theta} = \a(I) \cap \RR^{W_\Theta}$, what we want to show is that $\a(I)_\Theta = \a(I) \cap \RR^{W_\Theta}$. 

If $f \in \a(I)_\Theta$, then we can write $f=\der(Q)$ for some $\der \in D(\AA_I)^{W_\Theta}$. 
It is clear that $f=\der(Q) \in \a(I)$ since $D(\AA_I)^{W_\Theta} \subset D(\AA_I)$. 
We also see that $f=\der(Q) \in \RR^{W_\Theta}$ by \eqref{eq:der(Q)invariant} and hence one has $\a(I)_\Theta \subset \a(I) \cap \RR^{W_\Theta}$. 

Conversely, if we take $f \in \a(I) \cap \RR^{W_\Theta}$, then one can write $f=\der(Q)$ for some $\der \in D(\AA_I)$.
By setting $\der^\Theta \coloneqq \frac{1}{|W_\Theta|} \sum_{w \in W_\Theta} w \cdot \der \in (\Der \RR)^{W_\Theta}$, the derivation $\der^\Theta$ belongs to $D(\AA_I)$ from Proposition~\ref{proposition:WTheta-action}-(1). 
Hence, $\der^\Theta \in D(\AA_I)^{W_\Theta}$. 
We obtain $f = \der^\Theta(Q) \in \a(I)_\Theta$ by the following computation
\begin{align*}
\der^\Theta(Q) &= \frac{1}{|W_\Theta|} \sum_{w \in W_\Theta} (w \cdot \der)(Q) = \frac{1}{|W_\Theta|} \sum_{w \in W_\Theta} w(\der(w^{-1} Q)) = \frac{1}{|W_\Theta|} \sum_{w \in W_\Theta} w(\der(Q)) \\
&= \frac{1}{|W_\Theta|} \sum_{w \in W_\Theta} w(f) = \frac{1}{|W_\Theta|} \sum_{w \in W_\Theta} f =f.
\end{align*}
Therefore, we conclude that $\a(I)_\Theta = \a(I) \cap \RR^{W_\Theta} = \ker \varphi_{I, \Theta}$ as desired.
\end{proof}

\end{document}